\documentclass{amsart}
\usepackage{amssymb,cite}

\newcounter{maincounter}
\numberwithin{maincounter}{section}
\numberwithin{equation}{section}

\newtheorem{lemma}[maincounter]{Lemma}
\newtheorem{proposition}[maincounter]{Proposition}
\newtheorem{corollary}[maincounter]{Corollary}
\newtheorem{remark}[maincounter]{Remark}
\newtheorem{theorem}[maincounter]{Theorem}
\newtheorem{conjecture}[maincounter]{Conjecture}
\newtheorem{definition}[maincounter]{Definition}

\def\HH{\mathbb{H}}

\def\RR{\mathbb{R}}
\def\CC{\mathbb{C}}
\def\ZZ{\mathbb{Z}}

\def\QQ{\mathbb{Q}}
\def\FF{\mathbb{F}}
\def\GG{\mathbb{G}}

\def\AX{\citenum{AxSchanuel}}
\def\AY{\citenum{Ax72}} 
\def\BEZU{\citenum{BeZu01}} 
\def\BMZANOMALOUS{\citenum{BMZGeometric}} 
\def\BMZAA{\citenum{BMZUnlikely}} 
\def\DRIES{\citenum{D:oMin}} 
\def\DMM{\citenum{DMM:94}} 
\def\HABEGGERPILA{\citenum{HabeggerPila12}} 
\def\HARTSHORNE{\citenum{Hartshorne}} 
\def\KIRBY{\citenum{Kirby09}} 
\def\LANG{\citenum{Lang:transcendental}}
\def\LANGELLIPTIC{\citenum{Lang:elliptic}}
\def\MAHLER{\citenum{Mahler69}}

\def\PEST{\citenum{PS:04}}
\def\PILAOAO{\citenum{Pila:AO}}
\def\PILAEMS{\citenum{PilaEMS}}
\def\PILAWILKIE{\citenum{PilaWilkie}}  

\def\PINKPUB{\citenum{Pink05}}
\def\PINKPRE{\citenum{Pink}}    
\def\TSIMERMAN{\citenum{TsimermanBrauerSiegel}}
\def\ULLMO{\citenum{UllmoApp}}
\def\ULLMOYAFAEV{\citenum{UllmoYafaev11}}
\def\ZANNIERBOOK{\citenum{ZannierBook}}
\def\ZILBERSUMS{\citenum{Zilber}}

\newcommand{\cal}{\mathcal}

\newcommand{\hgtexp}{S}

\newcommand{\rank}{{\rm rank}\,}
\newcommand{\Hpoly}[2]{H^{}_{#1}({#2})}
\newcommand{\poly}[2]{#1^{}({#2})}
\newcommand{\polyt}[2]{#1^{\sim}({#2})}
\newcommand{\polytiso}[2]{#1^{\sim,{\rm iso}}({#2})}
\newcommand{\graph}[1]{\Gamma({#1})}

\newcommand{\IG}{\GG}
\newcommand{\IC}{\CC}
\newcommand{\IR}{\RR}
\newcommand{\IRan}{{\RR}_{\rm an}}
\newcommand{\IRanexp}{{\RR}_{\rm an,exp}}
\newcommand{\RRan}{\IRan}
\newcommand{\RRanexp}{\IRanexp}
\newcommand{\IRalg}{{\RR}_{\rm alg}}
\newcommand{\IQbar}{\overline{\QQ}}
\newcommand{\Kbar}{\overline{K}}
\newcommand{\IZ}{{\ZZ}}

\newcommand{\IQ}{\QQ}

\newcommand{\ts}[1]{{T}_0({#1})}

\newcommand{\cK}{{\mathcal K}}
\newcommand{\cL}{{\mathcal L}}
\newcommand{\cM}{{\mathcal M}}
\newcommand{\cO}{{\mathcal O}}
\newcommand{\cV}{{\mathcal V}}

\newcommand{\defZ}{Z}
\newcommand{\defF}{F}
\newcommand{\defW}{W}
\newcommand{\defC}{C}
\newcommand{\defE}{E}

\newcommand{\re}[1]{{\rm Re}({#1})}
\newcommand{\im}[1]{{\rm Im}({#1})}
\newcommand{\volS}{{\rm vol}}
\newcommand{\vol}[1]{\volS({#1})}
\newcommand{\orth}[1]{{#1}^{\bot}}
\newcommand{\mat}[2]{{\rm Mat}_{#1}({#2})}
\newcommand{\ssm}{\smallsetminus}

\newcommand{\opt}[2]{{\rm Opt}_{#2}({#1})}
\newcommand{\Height}[1]{{H}({#1})}
\newcommand{\trdeg}{{\rm trdeg\,}} 
\newcommand{\geo}[1]{\langle {#1}\rangle_{{\rm geo}}}
\newcommand{\defect}{\delta}
\newcommand{\geodef}{{\delta_{\rm geo}}}
\newcommand{\End}[1]{{\rm End}({#1})}
\newcommand{\Hom}[1]{{\rm Hom}({#1})}
\newcommand{\hommaxR}[1]{\text{\rm Hom}({#1})^{*}_{\IR}}
\newcommand{\arith}{\rm arith}
\newcommand{\sgu}[2]{{#1}^{[{#2}]}}
\newcommand{\oa}[1]{{#1}^{\rm oa}}
\newcommand{\codim}{{\rm codim}}
\newcommand{\lgo}{LGO}
\newcommand{\zcl}[1]{{\rm Zcl}({#1})}

\renewcommand{\subset}{\subseteq} 
\renewcommand{\supset}{\supseteq}

\def\grpx{{{X}}}
\def\grpy{{{Y}}}
\def\grpz{{{Z}}}

\def\var{{{V}}}

\def\opta{{{A}}}
\def\optb{{{B}}}
\def\optc{{{C}}}

\begin{document}
\title[O-minimality and intersections] {O-minimality and certain atypical
  intersections}
 

\author{P. Habegger and J. Pila}

\maketitle
\begin{abstract}
  We show that the strategy of point counting in o-minimal structures
 can be applied to various problems on unlikely intersections that go
  beyond the conjectures of Manin-Mumford and Andr\'e-Oort. 
We verify the so-called Zilber-Pink
Conjecture in a product of modular curves on assuming a lower bound
for Galois orbits and a sufficiently strong modular Ax-Schanuel
Conjecture. In the context of abelian varieties we obtain the
Zilber-Pink Conjecture  for curves unconditionally when everything is
defined over a number field. For higher dimensional subvarieties of
abelian varieties we
obtain  some weaker
results and some conditional results. 
\medskip

On d\'emontre que la strat\'egie de comptage dans des structures
o-minimales est suffisante pour traiter plusieurs probl\`emes qui vont
au-del\`a des conjectures de Manin-Mumford et Andr\'e-Oort. On
v\'erifie la conjecture de Zilber-Pink pour un produit de courbes
modulaires en supposant une minoration assez forte pour la taille de l'orbite de
Galois et en supposant une version modulaire du th\'eor\`eme de Ax-Schanuel.
Dans le cas des vari\'et\'es ab\'eliennes on d\'emontre la conjecture de Zilber-Pink 
pour les courbes 
si tous les objets sont d\'efinis sur un corps de nombres. Pour
les sous-vari\'et\'es de dimension sup\'erieure on obtient quelques 
r\'esultats plus faibles et quelques r\'esultats 
conditionnels. 
\end{abstract}
\tableofcontents

\bigbreak

\section{Introduction}
\label{sec:intro}


The object of this paper is to show that the ``o-minimality 
and point-counting'' strategy can be applied to quite general 
problems of ``unlikely intersection'' type as formulated in 
the Zilber-Pink Conjecture (ZP; see Section \ref{sec:ZP} for various formulations), 
provided one assumes
certain arithmetic and functional transcendence hypotheses.
In these problems there is an ambient variety $X$ of a certain 
type equipped with a distinguished collection ${\cal S}$ of 
``special'' subvarieties. The conjecture governs the intersections
of a subvariety $V\subset X$ with the members of ${\cal S}$. 
In the problems we consider,
$X$ will be either a product of non-compact modular curves 
(for which it is sufficient to consider the case $X=Y(1)^n$) 
or an abelian variety,
but the same formulations should be applicable more generally.
In this paper, a subvariety is always geometrically irreducible
and therefore in particular non-empty. 
A curve is a subvariety of dimension $1$.

Our most general results are conditional, but
let us state first an unconditional result in the abelian setting.

Say $\grpx$ is an abelian variety defined over a field
$K$ and $\Kbar$ is a fixed algebraic closure of $K$. For any $r\in\IR$
we write 
\begin{equation*}
\sgu{\grpx}{r}=  \bigcup_{\codim_{\grpx}H\ge r} H(\Kbar)
\end{equation*}
where $H$ runs over algebraic subgroups of $\grpx$ satisfying the
dimension condition.

\begin{theorem}
\label{thm:curvesabvar}
Let $\grpx$ be an abelian variety defined over a number field $K$
and suppose  $\var\subset \grpx$ is a curve, also defined over $K$. 
If $\var$ is not contained in a proper algebraic subgroup 
    of $\grpx$, then  $\var(\Kbar)\cap \sgu{\grpx}{2}$ is finite. 
\end{theorem}

This theorem is the abelian version of Maurin's Theorem
\cite{Maurin}. 
We will see a more precise version in
Theorem \ref{thm:abvar1}.

We briefly describe  previously known cases of 
 Theorem \ref{thm:curvesabvar}  under additional hypotheses on $\var$ or
 $\grpx$. 
Viada \cite{Viada} proved finiteness for $\var$ not contained in the
translate of a proper abelian subvariety and if  $\grpx$ is the power of
an elliptic curve with complex multiplication. 
R\'emond and Viada \cite{RV} then removed the hypothesis on $\var$.  
This was later generalised by Ratazzi \cite{Ratazzi08} to when the ambient
group variety
is isogenous to a power of an abelian variety with complex
multiplication. 
Carrizosa's height lower bound \cite{Carrizosa08, Carrizosa09}
in combination with R\'emond's height upper bound  \cite{RemondInterII}
led to a proof for all abelian varieties with complex
multiplication. Work of Galateau \cite{Galateau} and
Viada \cite{Viada08} cover the case of an arbitrary product of
elliptic curves.

\medskip

More generally we show, in the abelian and modular settings, that
the Zilber-Pink conjecture may be
reduced to two statements, one of an arithmetic nature,
the other a functional transcendence statement.
In general, the former statement remains conjectural in both settings. 
In the abelian setting, the functional transcendence statement 
follows from a theorem of Ax [\AY],
while in the modular setting a proof of it has been announced recently
by Pila-Tsimerman.
Both statements are generalisations of statements which have been
used to establish cases of the Andr\'e-Oort conjecture, 
and this aspect of our work is in the spirit of
Ullmo \cite{UllmoApp}.

The arithmetic hypothesis, 
which we formulate here, is the ``Large Galois Orbit'' hypothesis 
(LGO) and asserts that, for fixed  $V\subset X$,
certain (``optimal'') isolated intersection points
of $V$ with a special subvariety $T\in{\cal S}$ have
a ``large'' Galois orbit over a fixed finitely generated field of definition for
$V$, expressed in terms of a suitable complexity measure of $T$. 
Special subvarieties in our settings are described
in Section \ref{sec:ZP} 
for abelian varieties and Section \ref{sec:specialsubvarY1n} 
 for $Y(1)^n$
and LGO is formulated in Section \ref{sec:lgo}. 

In the context of the Andr\'e-Oort Conjecture, the Generalised
 Riemann Hypothesis (GRH) suffices to guarantee LGO 
(see \cite{TsimermanBrauerSiegel, UllmoYafaev11}). 
However, it is not clear to the authors if a variant of
the Riemann Hypothesis leads to large Galois orbits for isolated points
in $V\cap T$ if $\dim T \ge 1$. Indeed, in the Shimura setting,
there seems to be no galois-theoretic description of  isolated points
in $V\cap T$ which is rooted in class field theory.
On the other hand, suitable bounds have been established unconditionally
for Andr\'e-Oort in several cases, and perhaps LGO will be found
accessible without assuming GRH.

Associated with $X$ is a certain transcendental uniformisation
$\pi: U\rightarrow X$.
The functional transcendence hypothesis is the 
``Weak Complex Ax'' hypothesis (WCA) 
and is a weak form of an analogue for $\pi$
of ``Ax-Schanuel'' for cartesian powers of the exponential function. 
The latter result, due to Ax [\AX], affirms Schanuel's Conjecture 
(see [\LANG, p. 30]) for differential fields. WCA is formulated in
Section \ref{sec:wca}. 

In the modular case $X=Y(1)^n$ our result is the following.
A very special case of it was established unconditionally by us in 
[\HABEGGERPILA].

\begin{theorem}
\label{thm:LGOWCAimplyZP}
  If LGO and WCA hold for 
$Y(1)^n$ then the Zilber-Pink Conjecture  holds for subvarieties of
 $Y(1)^n$ defined over $\IC$.
Moreover, if WCA holds for $Y(1)^n$ and LGO holds
with $K=\IQ$, then 
the Zilber-Pink Conjecture  holds for subvarieties of
 $Y(1)^n$ defined over $\IQbar$.
\end{theorem}

In the case that $X$ is an abelian variety, 
we establish the same result in Section \ref{sec:unlikelyabvar}. 
However, as mentioned above, in this case WCA is known, 
and LGO can be established in the case that 
$V$ is one-dimensional when $X$ and $V$ are defined over 
$\IQbar$. This allows us to prove the above unconditional result for curves.

\medskip

All current approaches towards Theorem  \ref{thm:curvesabvar} require
a height upper bound on the set of points in question. 
Like many of the papers cited above we use R\'emond's height bound
\cite{RemondInterII} which relies on his generalisation of Vojta's height
inequality. 

In contrast to  previous approaches we do not rely on  delicate
Dobrowolski-type \cite{Carrizosa08, Carrizosa09, Ratazzi08} 
or Bogomolov-type \cite{Galateau} height lower bounds to pass from
bounded height to finiteness. 
These height lower bounds are expected (but not known) to
generalise to  arbitrary abelian varieties. 
Instead we will use a variation of the strategy originally
devised by Zannier to reprove the Manin-Mumford Conjecture \cite{PilaZannier} for
abelian varieties. This approach relied on the Pila-Wilkie
 point counting result in o-minimal structures. 
We will still require a height lower bound. However, the robust nature
of the method allows us to use Masser's general bound \cite{Masser:smallvalues} which
predates the sophisticated and essentially best-possible
results of Ratazzi and Carrizosa 
that require the  ambient abelian variety to have complex multiplication.

In her recent Ph.D. thesis, Capuano \cite{CapuanoPhD}  
counted rational points on suitable definable subsets of
a Grassmanian  to obtain finiteness results
on unlikely intersections with curves in the algebraic torus.  

In the next theorem we collect several partial results 
in the abelian setting for
subvarieties of arbitrary dimension. 

\begin{theorem}
\label{thm:higherdimabvar}
Let $\var\subset \grpx$ be  a subvariety  of an abelian variety,
both defined over a number field  $K$. 
Let us also fix an ample, symmetric line bundle on $\grpx$ and its
associated
N\'eron-Tate height $\hat h$. 
\begin{enumerate}
\item [(i)] 
If $\hgtexp\ge 0$ then
\begin{equation*}
  \left\{P \in \var(\Kbar)\cap \sgu{\grpx}{1+ \dim \var};\,\, \hat h(P)\le
  \hgtexp\right\}
\end{equation*}
is contained in a finite union of proper algebraic subgroups of $\grpx$. 
\item[(ii)]
Suppose
 $\dim \varphi(\var) = \min \{\dim \grpx/\grpy, \dim
  \var\}$ for all abelian subvarieties $\grpy\subset \grpx$ where
  $\varphi:\grpx\rightarrow \grpx/\grpy$ is the canonical morphism. If $\dim \var \ge
  1$ then 
  $\var(\Kbar)\cap \sgu{\grpx}{1+\dim \var}$ is not Zariski dense in $\var$. 
\end{enumerate}
\end{theorem}

Theorem \ref{thm:abvar2}  refines this statement.
A particularly simple example of a surface that does not satisfy the
hypothesis in (ii) is the square $\var\times \var$ of a curve 
$\var\subsetneq\grpx$. 
The Zilber-Pink Conjecture remains open for surfaces in
abelian varieties defined over a number field. 
In Corollary \ref{cor:unlikelyabvarQbar} we show that a sufficiently
strong height upper bound leads to a proof of the Zilber-Pink
Conjecture for abelian varieties over number fields. 


\medskip

The burden of the theorems
is that the o-minimality/point-counting strategy is 
adequate to  deal with ``atypical intersections'', 
given these additional ingredients. 
To some extent this is already 
demonstrated for curves by the work of Masser and Zannier 
\cite{MZ:torsionanomalous} and the authors' earlier work [\HABEGGERPILA]. 
Here we will handle subvarieties of arbitrary dimension and
confirm that the o-minimal method scales to this generality.
Ullmo [\ULLMO] shows that a  ``large Galois orbit'' statement, ``Ax-Lindemann'', 
and a height upper bound for certain preimages of special points, 
together with a suitable definability result, 
enable a proof of the Andr\'e-Oort Conjecture by point-counting and
o-minimality. 
In our setting, the special subvarieties generally have positive
dimension and are often unaffected by the Galois action. Rather we
must count objects that arise when  intersecting them with the given subvariety.
More generally, one can formulate results along the lines of the Counting Theorem 
of [\PILAWILKIE] for  ``atypical intersections'' of a definable set  in an o-minimal 
structure with linear subvarieties defined over $\QQ$ (or indeed the
members of any definable family of subvarieties having rational parameters). 
Here (as in previous results by these methods) o-minimality 
is used for more than point-counting.

For the basics about o-minimality see [\DRIES].
The definability properties required in this paper are 
afforded by the result, 
due to Peterzil-Starchenko [\PEST], 
that the $j$-function restricted to the usual 
fundamental domain $\FF_0$
for the ${\rm SL}_2(\ZZ)$ 
action on the upper half plane $\HH$ is definable 
in the o-minimal structure $\RRanexp$
(see [\DMM]).  Accordingly, in this paper ``definable'' 
will mean ``definable in $\RRanexp$'' unless stated otherwise. 
However, the exponential function is superfluous when working with abelian
  varieties.
Here it is enough to work in $\RRan$, the
structure generated by  restricted real analytic functions,
which was recognised as being o-minimal by van den Dries after
work of Gabrielov.
In Section \ref{sec:semirational} we will work with more general
o-minimal structures. 

\section{The Zilber-Pink Conjecture}
\label{sec:ZP}


The Zilber-Pink Conjecture is a far-reaching generalisation of
the Mordell-Lang and Andr\'e-Oort Conjectures. Different
formulations in different settings, but based on the same underlying 
idea of ``atypical'' or ``unlikely'' intersections, 
were made by Zilber [\ZILBERSUMS], 
Pink [\PINKPUB, \PINKPRE], and
Bombieri-Masser-Zannier [\BMZANOMALOUS]. 

Bombieri-Masser-Zannier [\BMZAA] show that all
the versions are equivalent for  $\GG_m^n$. 
Essentially the same argument shows they are equivalent for $Y(1)^n$, 
but in the general case this is unclear. 
The version we give is in any case the 
strongest: essentially it is the Zilber or 
Bombieri-Masser-Zannier statement 
in Pink's general setting,
where the ambient variety is a mixed Shimura variety (see [\PINKPRE]).
For an introduction to the conjecture and the state-of-the-the-art
see [\ZANNIERBOOK].

The general setting involves and ambient variety which
is a {\it mixed Shimura variety\/} (see [\PINKPUB]).
A mixed Shimura variety $X$ is endowed 
with a (countable) collection  ${\cal S}={\cal S}_X$ of 
{\it special subvarieties\/} and a larger (usually uncountable) collection 
${\cal W}={\cal W}_X$ of {\it weakly special\/} subvarieties. 
Special subvarieties of dimension zero are called 
{\it special points.\/}

Say $X$ is an algebraic torus $\IG_m^n$, an abelian variety, or even a
semi-abelian variety defined over $\IC$. 
In some cases  $X$ appears as a special subvariety of a family of
semi-abelian varieties and may be treated as a mixed Shimura variety. 
In general, 
 $X$ shares enough formal 
 properties 
with mixed Shimura varieties to formulate a Zilber-Pink Conjecture on
unlikely intersections.
A coset of $X$ is the translate of a connected algebraic subgroup of $X$.
By abuse of notation we also call cosets  weakly special subvarieties
of $X$.
A torsion coset of $X$ is a coset of $X$ that contains a point of
finite order. We sometimes call torsion cosets special subvarieties of
$X$. We also write ${\cal S}_X$ and ${\cal W}_X$ for the set of all
cosets and torsion cosets of $X$, respectively.
So special points are torsion points. 
We study at torsion cosets
in more detail in \ref{sec:degabsv}.

\begin{definition}
\label{def:atypical}
  Let $X$ be a mixed Shimura variety or a semi-abelian variety defined
  over $\IC$.
Let ${\cal S}$ be the collection of special subvarieties of $X$.
Let $V\subset X$ be a subvariety.
A subvariety $A\subset V$ is called {\it atypical\/}  (for $V$ in $X$)
if there is a special subvariety $T$ of $X$ such that $A$ is a 
component of  $V\cap T$ and
$$
\dim A>\dim V+\dim T-\dim X
$$
i.e. $A$ is an ``atypical component'' in Zilber's terminology [\ZILBERSUMS]. The 
{\it atypical set\/} of $V$ (in $X$)  is the union of all 
atypical subvarieties, and is denoted
${\rm Atyp\/} (V, X)$.
\end{definition}

A special subvariety of a mixed Shimura variety is itself
a mixed Shimura variety, with
$$
{\cal S}_T=\{A;\,\, A{\rm\ is\  a \ component\ of\ } T\cap S
{\rm\ for\ some\ } S\in {\cal S}_X\},
$$
$$
{\cal W}_T=\{A;\,\, A{\rm\ is\  a \ component\ of\ } T\cap S
{\rm\ for\ some\ } S\in {\cal W}_X\}.
$$
The following is a ``CIT'' (cf. Conjecture 2 \cite{Zilber}) version of Pink's Conjecture, 
and is on its face stronger than the statement conjectured by Pink. 
It is convenient to frame it as a statement about all special 
subvarieties of a given mixed Shimura variety $X$.


\begin{conjecture}[Zilber-Pink (ZP) for $X$]
\label{conj:ZPatypical}
  Let $X$ be a mixed Shimura variety or a semi-abelian variety defined
  over $\IC$.
  Let $T\in{\cal S}_X$ 
and $V\subset T$. 
Then ${\rm Atyp\/} (V, T)$ is a finite union of atypical subvarieties.
Equivalently, $V$ contains only a  finite number of maximal 
atypical (for $V$ in $T$) subvarieties.
\end{conjecture}

Now some subvarieties are more atypical than others.
Since the collection of special subvarieties is closed under taking 
irreducible  components of intersections and contains $X$, given $A\subset X$ 
there is a smallest special 
subvariety containing $A$ which we denote $\langle A\rangle$. 
We abbreviate
$\langle P\rangle=\langle \{P\}\rangle$ for a singleton $\{P\}$. 
We define (following Pink [\PINKPRE]) the {\it defect\/} of $A$ as
\begin{equation}
\label{eq:defdefect}
\delta(A)=\dim \langle A\rangle-\dim A.
\end{equation}
Then $A\subset V$ is atypical for $V$ in $X$ if $\delta(A)<\dim X-\dim V$. 


The atypical set is simply a union of atypical subvarieties of $V$, 
and it may 
happen  that an atypical subvariety is contained in some larger but less 
atypical subvariety. A generalisation of the argument showing
that ZP implies the Andr\'e-Oort Conjecture (which 
corresponds to subvarieties of defect 0) 
shows that ZP implies a notionally stronger
version in which one considers subvarieties of each defect separately.

For $T\in {\cal S}_X$, $V\subset T$, and a positive integer 
$\delta$ denote by
\begin{equation}
\label{eq:defatypical2}
{\rm Atyp\/}^\delta(V, T)
\end{equation}
the union of subvarieties $A\subset V$ with $\delta(A)\le \delta$.
Note that this does not depend on $T$ since under our assumptions
$\langle V\rangle \subset T$.

\begin{conjecture}[Articulated Zilber-Pink (AZP) for $X$]
\label{conj:AZP}
  Let $X$ be a mixed Shimura variety or a semi-abelian variety defined
  over $\IC$.
   Let $T\in{\cal S}_X$, $V\subset T$, 
and $\delta$ a non-negative integer. 
Then ${\rm Atyp\/}^\delta (V, T)$ is a finite union of 
atypical subvarieties of defect $\le \delta$.
\end{conjecture}

Since the components of $V$ have defect at most $\dim X-\dim V$, 
the conjecture is trivial for $\delta\ge \dim X-\dim V$.
The following implication uses only the formal properties 
of special subvarieties; the reverse implication is immediate.
A version of this result appears in [\ZILBERSUMS].

\begin{proposition}
\label{prop:ZPimpliesAZP}
  Let $X$ be as in Conjecture \ref{conj:AZP}. Then 
ZP for $X$ implies AZP for $X$.
\end{proposition}
\begin{proof}
  By induction on the dimension of the ambient special 
subvariety  $T=\langle V\rangle \subset X$. AZP is clear if $\dim T=0$. 
So we assume now that ZP holds for $T$ while AZP holds
for all proper special subvarieties of $T$. Let $V\subset T$. 
By ZP, there are finitely many atypical subvarieties $A_i$, 
with associated special subvarieties $T_i$,
such that every atypical subvariety $A$ is contained in some $A_i$. 
The $T_i$
are evidently proper ($V\subset T$ is not atypical for $V$ in $T$). 
Now fix $\delta$. Then with $\delta_i=\delta(A_i)$,
$$
{\rm Atyp}^\delta(V, T)=\big(\bigcup_{i: \delta_i\le \delta} A_i\big) \cup 
\big(\bigcup_{i: \delta_i>\delta} {\rm Atyp}^\delta(A_i, T_i)\big)
$$
is a finite union by the induction hypothesis, which gives AZP for the
$A_i$. 
\end{proof}


Let us formulate one last conjecture in the same spirit as those
above. 

\begin{definition}
  Let $X$ be a mixed Shimura variety or a semi-abelian variety defined
  over $\IC$.
  Let $V\subset X$ be a subvariety. A subvariety $A\subset V$ is said to
be  optimal if there is no subvariety $B$ with
$A\subsetneq B\subset V$ such that
$$
\delta(B)\le \delta(A).
$$
The specification of $V$ and $X$ will generally be suppressed, as no confusion
should arise.
We write $\opt{\var}{}$ for the set of all optimal subvarieties of $\var$.
\end{definition}

It is illustrative to consider some
formal properties of optimal subvarieties $\opta\in\opt{\var}{}$.
Clearly, $\opta$ 
 is an irreducible component of
$\var\cap \langle \opta\rangle$. 
The whole subvariety is always optimal, i.e. $\var\in\opt{\var}{}$.
So  ``maximal optimal subvariety'' is a  useless concept.
In a certain sense maximality is built into the notation of optimality.
Indeed, an optimal subvariety of defect $0$  is a maximal torsion coset  contained completely in $\grpx$.
If $\opta\not=\var$, then $\defect(\opta)  <\defect(\var)$
or in other words
\begin{equation*}
\dim \opta > \dim\langle\opta\rangle + \dim \var - \dim \langle\var\rangle.
\end{equation*}
So $\opta$ is atypical for $\var$ in $\langle\var\rangle$. 
We will see that the arithmetically  interesting case is when $\opta=\{P\}$ is a
singleton.
Then $P$ is contained in a special subvariety of dimension strictly
less than $\dim\langle V\rangle-\dim V$. 

\begin{conjecture}
\label{conj:ZPoptimal}
  Let $X$ be a mixed Shimura variety or a semi-abelian variety defined
  over $\IC$. Then $X$ contains only finitely many optimal
  subvarieties. 
\end{conjecture}

\begin{lemma}
  Let $X$ be a mixed Shimura variety or a semi-abelian variety defined
  over $\IC$. Then the conclusions of Conjectures \ref{conj:ZPatypical}
  and \ref{conj:ZPoptimal} are equivalent for $X$.
\end{lemma}
\begin{proof}
Let $\var$ be a subvariety of $X$.

  First we suppose that $\opt{\var}{}$ is finite. 
Let $T$ be a special subvariety of $\grpx$ containing $\var$.
We may assume $\langle\var\rangle=T$, as ${\rm Atyp\/}{(V,T)}=V$ otherwise.
Let $\opta$ be atypical for $\var$ in $T$
and let $\optb\supset \opta$ be an optimal subvariety of $\var$
with $\defect(\optb)\le\defect(\opta)$. 
Then $\defect(\optb)< \dim T-\dim \var$ and so 
$\optb$ is also atypical for $\var$ in $T$.
Conjecture \ref{conj:ZPatypical} follows for $\grpx$
as $\opta$ is contained in a member of the finite set 
$\opt{\var}{}\ssm\{V\}$ of proper optimal subvarieties of $V$.

Let us now assume conversely that  Conjecture \ref{conj:ZPatypical}
 holds for $\var$ with $T=\langle \var\rangle$. 
We must show that there are only finitely many possibilities for
 $\opta\in \opt{\var}{}$. Clearly, we may assume $\opta\subsetneq\var$. 
But then $\defect(\opta) < \defect(\var) = \dim T - \dim \var$ by
 optimality. 
So $\opta$ is also atypical for $\var$ in $T$. 
It is contained in a subvariety $\optb$ that is maximal atypical for
 $\var$ in $T$. 
As $\optb$ comes from a finite set we are done if $\opta=\optb$. 
So say $\opta\subsetneq\optb$.
We observe $\optb\not=\var$ and $\opta$ lies in $\opt{\optb}{}$, which is
 finite by induction on the dimension.
\end{proof}

A product of modular curves,
in particular $Y(1)^n$, is a (pure) Shimura variety. Its special
subvarieties are described in detail in Section \ref{sec:specialsubvarY1n}.

Another example of a (pure) Shimura variety is the moduli space
${\cal A}_g$ of principally polarised abelian varieties of dimension $g$.
It is a special subvariety
of a larger mixed Shimura variety ${\cal X}_g$ which consists
of ${\cal A}_g$ fibered at each point by the 
corresponding abelian variety.


Conjecture  \ref{conj:ZPatypical} 
for an abelian variety $X$ and its special subvarieties is equivalent
to ZP as formulated for subvarieties $V\subset X\subset {\cal X}_g$
(see [\PINKPRE, 5.2] where the equivalence is proved for Pink's 
Conjecture; with obvious modifications the argument proves
the equivalence in the version we have given).

\begin{remark}
  With the same definitions, the same equivalence 
presumably holds for any weakly special subvariety of a 
mixed Shimura variety. It seems interesting to ask under which
conditions the same holds for a subvariety of a mixed
Shimura variety.
\end{remark}

\section{Special subvarieties}
\label{sec:specialsubvar}

In the next two sections we discuss in more detail the special
subvarieties of an abelian variety and of $Y(1)^n$. 

\subsection{The abelian setting}
\label{sec:degabsv}

In the case of abelian varieties, the special subvarieties are
the torsion cosets, i.e. the
irreducible components of algebraic subgroups or in other words, 
 translates of abelian subvarieties by points of finite order.
In this section we recall some basic facts on torsion cosets and we
define their complexity. 

An inner-product on an $\IR$-vector space $W$ is a symmetric, positive
definite bilinear form $\langle \cdot,\cdot \rangle:W\times
W\rightarrow \IR$. The volume $\vol{\Omega}$  of  a  finitely generated
subgroup $\Omega$ of
$W$ with respect to $\langle\cdot,\cdot\rangle$ is 
$\vol{\Omega} = |\det(\langle\omega_i,\omega_j\rangle)|^{1/2}$
for any $\IZ$-basis
 $(\omega_1,\ldots,\omega_\rho)$  of $\Omega$. 
The volume is independent of the choice of basis.
The orthogonal complement of a vector subspace $U\subset W$ is
$\orth{U} =\{w\in W;\,\,\langle w,u\rangle=0\text{ for all $u\in
  U$}\}$.

Let $\grpx$ be an
 abelian variety defined over $\IC$ with
$\dim \grpx = g\ge 1$ and suppose that $\cL$ is an ample line bundle on $\grpx$. 
The degree of $\grpx$ with respect to $\cL$ is the intersection number
$\deg_\cL \grpx = (\cL^{g}[\grpx]) \ge 1$. 

The  line bundle $\cL$ defines a hermitian form
\begin{equation*}
  H:\ts{\grpx}\times\ts{\grpx}\rightarrow \IC
\end{equation*}
on the tangent space $\ts{\grpx}$ of $\grpx$ at the origin.
This form is positive definite since $\cL$ is ample.
It is $\IC$-linear in the first argument and satisfies
$H(v,w) =\overline{H(w,v)}$ for $v,w\in\ts{\grpx}$. 
The real part $\re{H}$ is an inner-product
$\langle\cdot,\cdot\rangle$
on $\ts{\grpx}$ taken as an $\IR$-vector space of dimension $2g$. 
Thus we obtain a norm
$\|v\| = \langle v,v\rangle^{1/2}$ on $\ts{\grpx}$. 
The imaginary part $E=\im{H}$ is a non-degenerate symplectic form 
$V\times V \rightarrow\IR$.

Let $\Omega_\grpx\subset \ts{\grpx}$ denote the period lattice of $\grpx$. It is a
free
abelian group of rank $2g$ and generates $\ts{\grpx}$ as an
$\IR$-vector space. Therefore, $\volS(\Omega_\grpx) > 0$
where $\volS$ denotes the volume with respect to the inner-product
$\langle\cdot,\cdot\rangle_\cL$. 
The subgroup $\Omega_\grpx$ is discrete in $\ts{\grpx}$.

\begin{lemma}
\label{lem:deg}
  We have $\deg_\cL \grpx = g! \vol{\Omega_\grpx}$.
\end{lemma}
\begin{proof}
This is a well-known consequence of the Riemann-Roch Theorem for
abelian varieties, see Chapter 3.6 \cite{CAV}.
\end{proof}

Now suppose that $\grpy\subset \grpx$ is an abelian subvariety of dimension
$\dim\grpy \ge 1$. The pull-back
$\cL|_\grpy$ of $\cL$ by the inclusion map $\grpy\hookrightarrow
\grpx$ is an ample line bundle on  $\grpy$. 
We treat $\ts{\grpy}$ as a vector subspace of $\ts{\grpx}$. The hermitian form 
on $\ts{\grpy}$ induced by $\cL|_\grpy$ is  just the restriction of $H$. 
Let $\Omega_\grpy \subset \ts{\grpy}$ denote the period lattice of $\grpy$. 
Then $\deg_{\cL|_\grpy}{\grpy} = (\dim\grpy)! \vol{\Omega_\grpy}$ by the previous
lemma. The projection formula implies
$\deg_{\cL|_\grpy} \grpy = (\cL^{\dim \grpy}[\grpy])$. We will abbreviate
this degree by $\deg_\cL \grpy$.

The next lemma uses Minkowski's Second Theorem from the Geometry of
Numbers. 
\begin{lemma}
\label{lem:geonumbs}
There exists a constant $c>0$ depending only on $(\grpx,\cL)$ with the
following properties. 
  \begin{enumerate}
  \item [(i)]
There exist linearly independent periods
$\omega_1,\ldots,\omega_{2\dim\grpy}\in \Omega_\grpy$ with
$  \|\omega_i\|\le 
c \deg_\cL{\grpy}$ for $1\le i \le 2\dim\grpy$
and $\|\omega_1\|\cdots\|\omega_{2\dim\grpy}\|\le c\deg_\cL{\grpy}$.
\item[(ii)] If $z\in \Omega_\grpx + \ts{\grpy}$ there exist
$\omega\in \Omega_\grpx$ with $z-\omega\in \ts{\grpy}$
 and 
$\|\omega\|\le \|z\| + c\deg_\cL \grpy$. 
  \end{enumerate}
\end{lemma}
\begin{proof}
  Let $0<\lambda_1\le \ldots\le \lambda_{2\dim\grpy}$ be the successive minima of
  $\Omega_\grpy$ with respect to the closed unit ball $\{z\in \ts{\grpy};\,\,
  \|z\|\le 1\}$. By Minkowski's Second  Theorem we
 have 
 \begin{equation}
\label{eq:minkowski2}
   \lambda_1\cdots \lambda_{2 \dim\grpy}\le 2^{2 \dim\grpy} 
\frac{\vol{\Omega_\grpy}}{\mu(2 \dim\grpy)}
 \end{equation}
where $\mu(n) > 0$ 
 denotes the Lebesgue volume
of  the unit ball
in $\IR^n$. 
There exist independent elements
$\omega_1,\ldots,\omega_{2\dim\grpy}\in \Omega_\grpy$ with $\|\omega_i\|\le \lambda_i\le
\lambda_{2 \dim\grpy}$. 
Let $\rho=\rho(\grpx,\cL) > 0$ denote the minimal norm of a period 
in $\Omega_\grpx$. 
Using (\ref{eq:minkowski2}) and $\lambda_i \ge \rho $ we
estimate
\begin{equation*}
  \|\omega_i\|\le \lambda_{2\dim\grpy} \le
\frac{2^{2 \dim\grpy} }{\mu(2\dim Y) \rho^{2\dim\grpy-1}}
\vol{\Omega_\grpy}.
\end{equation*}
The first inequality of (i) follows from Lemma \ref{lem:deg}  applied
to $\grpy$. 
The second inequality follows easily from (\ref{eq:minkowski2}). 

Now say $z=\omega'+y$ is as in part (ii)
where $\omega'\in \Omega_\grpx$ and $y\in \ts{\grpy}$. 
The  periods $\omega_1,\ldots,\omega_{2\dim\grpy}$
 generate $\ts{\grpy}$ as an $\IR$-vector space. So
$y = \alpha_1\omega_1+\cdots + \alpha_{2\dim\grpy}\omega_{2\dim\grpy}$
for some $\alpha_1,\ldots,\alpha_{2\dim\grpy}\in\IR$. For each $i$ we fix
$a_i\in\IZ$ with $|\alpha_i-a_i|\le 1/2$. Then
$\omega'' = a_1\omega_1+\cdots + a_{2\dim\grpy}\omega_{2\dim\grpy}$ is
a period of $X$ and 
\begin{equation*}
  z = \omega'+\omega'' + \sum_{i=1}^{2\dim\grpy} (\alpha_i-a_i) \omega_i. 
\end{equation*}
Part (ii) follows with $\omega=\omega'+\omega''$,
the inequalities from (i), and the triangle inequality.
\end{proof}

Replacing $\cL$ by another ample line bundle leads to a notion 
degree that is comparable to the old one. 
\begin{lemma}
\label{lem:equivdeg}
  Let $\cM$ be an ample line bundle on $\grpx$. There exists
a constant $c \ge 1$ depending only on $\grpx,\cL,$ and $\cM$ but not on
$\grpy$ such that
$c^{-1} \deg_\cL \grpy \le \deg_\cM \grpy\le c\deg_\cL \grpy$.
\end{lemma}
\begin{proof}
To distinguish the norms and volumes coming from both line bundles
we write  $\|\cdot\|_\cL$, $\|\cdot\|_\cM$
and $\volS_\cL$, $\volS_\cM$.
Let $\omega_1,\ldots,\omega_{2\dim\grpy}$ be as in
Lemma \ref{lem:geonumbs}(i), so
$\|\omega_1\|_\cL\cdots \|\omega_{2\dim\grpy}\|_\cL
\le c\deg_\cL{\grpy}$. 
As all norms on $\ts{\grpx}$ are equivalent there exists
$c'\ge 1$ with $\|v\|_\cM \le c'\|v\|_\cL$ for all
$v\in\ts{\grpx}$.
So $\|\omega_1\|_\cM \cdots \|\omega_{2\dim\grpy}\|_\cM
\le {c'}^{2\dim\grpy}\|\omega_1\|_\cL \cdots \|\omega_{2\dim\grpy}\|_\cL$
and
Hadamard's inequality implies 
\begin{equation*}
  \volS_{\cM}(\Omega_\grpy) \le 
\|\omega_1\|_\cM \cdots \|\omega_{2\dim\grpy}\|_\cM
\le 
c {c'}^{2\dim\grpy} \deg_\cL{\grpy}. 
\end{equation*}
The second inequality  in the  lemma follows from $\volS_{\cM}(\Omega_\grpy) =
(\deg_\cM \grpy)/(\dim \grpy)!$.
The first one follows by symmetry.
\end{proof}


\begin{definition}
\label{def:complexitytorsioncoset}
If  $\opta$ is a torsion coset of $\grpx$ which is the translate of an
 abelian subvariety $\grpy$ of $\grpx$ by a torsion point,
we define its arithmetic complexity as
  \begin{alignat*}1
\Delta_{\arith}(\opta) &= 
\min\{ \text{order of } T;\,\, \opta=T+\grpy\text{ and $T$ has
    finite order}\}
  \end{alignat*}
and its complexity as
  \begin{alignat*}1
    \Delta(\opta) &= \max\big\{\Delta_{\arith}(\opta),\deg_\cL{\grpy}\big\}\ge 1
  \end{alignat*}
where $\deg_\cL{\grpy}$ is the degree of $\grpy$ with respect to $\cL$.   
\end{definition}

We do not emphasise the choice of $\cL$ in the complexity.
According to Lemma \ref{lem:equivdeg} changing the line bundle leads
to an arithmetic complexity which is comparable to the original one
up to a controlled factor. For our application it is enough to fix
once and for all a line bundle on the ambient abelian variety.

\subsection{The modular setting}
\label{sec:specialsubvarY1n}

In this section we describe the special subvarieties 
of $Y(1)^n$ together with some additional definitions and notations that will be used in the sequel.

Let $j:\HH\rightarrow Y(1)$ denote the $j$-function. 
By $\pi$ we denote the cartesian power of this map
$$
\pi:\HH^n\rightarrow Y(1)^n.
$$

Two-by-two real matrices with positive determinant act on 
$\HH$ by fractional 
linear transformations. If $g\in{\rm GL}_2^+(\QQ)$ then the functions $j(z)$ 
and $j(gz)$ on $\HH$ are related by a modular polynomial
$$
\Phi_N(j(z), j(gz))=0
$$
for a suitable positive integer $N=N(g)$ (in fact $N(g)$ is the determinant if $g$
is scaled to have relatively prime integer entries; see [\LANGELLIPTIC, Ch. 5, \S2]).

\begin{definition}
  A  strongly special curve in $\HH^n$ is the image of a 
map of the form
$$
\HH\rightarrow\HH^n,\quad  z\mapsto (g_1z, \ldots, g_{n}z)
$$
where $g_1=1, g_2\ldots, g_{n}\in{\rm GL}_2^+(\QQ)$.
\end{definition}

By a {\emph{strict partition}} we will mean a partition in which one 
designated part only is permitted to be empty.
\begin{definition}
\label{def:wss}
  Let $R=(R_0, R_1,\ldots, R_k)$ be a strict partition
of  $\{1,\ldots,n\}$ in which $R_0$ is permitted to be empty 
(and $k=0$  is permitted). For each index $j$ we let $\HH^{R_j}$ 
denote the corresponding cartesian product. 
A {\it weakly special subvariety\/} (of type $R$) of $\HH^n$ is a product
$$
Y=\prod_{j=1}^k Y_j
$$
where $Y_0\in \HH^{R_0}$ is a point and, for $j=1,\ldots, k$, 
$Y_j$ is a strongly special curve in $\HH^{R_j}$. We have $\dim Y=k$.
\end{definition}

\begin{definition}
A weakly special subvariety is called a  special subvariety
if each coordinate of $Y_0$ is a quadratic point of $\HH$.  
\end{definition}

With a quadratic $z\in\HH$ we associate its {\it discriminant\/} 
$\Delta(z)$, namely 
$\Delta(z)=b^2-4ac$ where $aZ^2+bZ+c$ is the minimal 
polynomial for $z$ over $\ZZ$ with $a>0$. 

\begin{definition}
For a special subvariety $Y$ as above we define the
{\it complexity\/}
$$
\Delta(Y)=\max \big(\Delta(z), N(g) \big)
$$
over all the coordinates $z$ of $Y_0$ and all $g\in{\rm GL}_2^+(\QQ)$ involved 
in the definition of the $Y_i, i\ge 1$.
\end{definition}

Note that a weakly special subvariety has a certain number of 
``non-special conditions'', namely the number of coordinates of $Y_0$ 
which are not quadratic,  and is special just if this number is zero.

Further, weakly special subvarieties come in families. 
Given a strict partition
$R=(R_0, R_1,\ldots, R_k)$ we may form a new strict partition $S$
in which the elements
previously in $R_0$ are made into individual parts, the parts 
$R_1,\ldots, R_k$ are retained, but $S_0$ is empty. 
Now a weakly special 
subvariety  $W$ of type $R$ comes with a choice of some elements 
$g_i\in{\rm GL}_2^+(\QQ)$ for the parts $R_i$. This same
choice determines a weakly special subvariety $T$ of type $S$ 
which is in fact 
special (even ``strongly special'', as there are no fixed coordinates).
We call $T$ the {\it special envelope\/} of type $R$.
The variety $W$ now lies in the family of weakly special 
subvarieties of $T$
corresponding to choices for the fixed coordinates $R_0$. 
It is thus a family
of weakly special subvarieties of $T$ parameterised by $\HH^{R_0}$. 
We will call the members of the family  {\it translates\/} of the 
{\it strongly special subvariety\/} of $\HH^{R_1\cup\ldots\cup R_k}$ 
corresponding to the given elements $g_i$,
the space of translates being $\HH^{R_0}$. The translate of $T$ by
$t\in \HH^{R_0}$ we denote $T_t$. 

We apply the same terminology to the images in $Y(1)^n$. 
In particular, we have the following.

\begin{definition}
  A  weakly special subvariety of $Y(1)^n$ 
is the image   $j(Y)$ where $Y$ is a weakly special subvariety 
of $\HH^n$, and is   special if $Y$ is 
special (for some or equivalently all possible choices for $Y$).
The  complexity of a special subvariety $T\subset Y(1)^n$, denoted 
$\Delta(T)$, is equal to the complexity of $Y$ (any choice will give
the same complexity due to the ${\rm SL}_2(\ZZ)$ invariance).
\end{definition}

As observed the weakly special subvariety $Y\subset \HH^n$ 
is a fibre in a family of weakly special subvarieties of some 
special subvariety $T$.  Thus, the image 
under $j$ of $Y$ and the other translates are algebraic 
subvarieties of $j(T)$.
 
\section{Geodesic-optimal subvarieties}
\label{sec:optimal}


Throughout this section and if not further specified let $X$ be a mixed Shimura variety or a semi-abelian variety defined
over $\IC$. 

The collection of 
weakly special  subvarieties, like the collection of special subvarieties, is closed under 
taking irreducible components
of intersection and contains the ambient variety, so there is a smallest 
{weakly special subvariety} 
$\langle A\rangle_{\rm geo}$ 
containing any subvariety  $A\subset X$.  We denote by
$$
\delta_{\rm geo}(A)=\dim \langle A\rangle_{\rm geo}-\dim A
$$
the {\it geodesic defect\/} of $A$.

\begin{definition}
  Let $V\subset X$ be a subvariety. A subvariety $A\subset V$ is said to
be  geodesic-optimal (for $V$ in $X$) if there is no subvariety 
$B$ with
$A\subset B\subsetneq V$ such that
$$
\delta_{\rm geo}(B)\le \delta_{\rm geo}(A).
$$
As for the defect, the specification of $V$ and $X$ will generally be 
suppressed.
\end{definition}

Note: what we call ``geodesic-optimal'' has been termed ``cd-maximal''
(co-dimension maximal) in the multiplicative context by Poizat \cite{Poizat}; 
see also \cite{BHMW}.

Again, if $A\subset V$ is geodesic-optimal then it is a component of 
$V\cap \langle A\rangle_{\rm geo}$.
Further, as
special subvarieties are weakly special, we have, for any 
$V$ and $A\subset V$, 
$\langle A\rangle_{\rm geo}\subset \langle A\rangle$ and so 
$\delta_{\rm geo}(A)\le\delta(A)$.
In contrast to the defect, the geodesic defect of a singleton is
always $0$. Therefore, a singleton is geodesic-optimal in $V$ if and only if
it is not contained in a coset of positive dimension contained in $V$. 


\begin{definition}
\label{def:defcondition}
  We say that $X$ has the  defect condition if
for any subvarieties $ A\subset B\subset X$ we have
$$
\defect(B)-\geodef(B)\le \defect(A)-\geodef(A).
$$
\end{definition}

\begin{proposition}
\label{prop:defectcondition}
  The defect condition holds 
  \begin{enumerate}
\item[(i)] if $\grpx=\IG_m^n$ is an algebraic torus, 
\item[(ii)] or if $\grpx$ is an abelian variety,    
\item [(iii)] or if $X=Y(1)^n$.
  \end{enumerate}
\end{proposition}
\begin{proof}
Let $A\subset B\subset \grpx$ be as in Definition \ref{def:defcondition}.
For (i)  let $B\subset \GG_{\rm m}^n$ and  let
\begin{alignat*}1
L&=\{(a_1,\ldots, a_n)\in\ZZ^n;\,\, x_1^{a_1}\cdots x_n^{a_n}
{\rm\ is\ constant\ on\ \/} B\},\\
M&=\{(a_1,\ldots, a_n)\in\ZZ^n;\,\, x_1^{a_1}\cdots x_n^{a_n}
{\rm\ is\ constant\ and\ a\ root\ of\ unity\ on\ \/} B\}
\end{alignat*}
be free abelian groups. 
Then ${\rm codim}\langle B\rangle=\rank M$ and
${\rm codim}\langle B\rangle_{\rm geo}=\rank L$, so that 
$$
\defect(B)-\geodef(B)=\rank L/M
$$
is the multiplicative rank of constant monomial functions on $B$.
Such functions remain constant and multiplicatively independent 
on  $A$.

To prove (ii) let $\opta$ and $\optb$ be subvarieties of $\grpx$. 
The coset $\geo{\optb}$ is a translate of an abelian subvariety
$\grpy$ of $\grpx$. Let us write $\varphi:\grpx\rightarrow \grpx/\grpy$ for the quotient
morphism. We fix an auxiliary point $P\in \opta(\IC)$. 

We remark that $\langle P \rangle + \grpy$ is a torsion 
coset that contains $P+\grpy$. 
As $P\in \optb(\IC)$ we also have
$\optb\subset P+\grpy$ and thus 
$\langle \optb\rangle \subset \langle P \rangle + \grpy$. We apply $\varphi$,
which has kernel $\grpy$, to find
$\varphi(\langle \optb\rangle)\subset \varphi(\langle P\rangle)\subset
\varphi(\langle \opta\rangle)$.
We conclude
\begin{equation}
\label{eq:dimbound}
\dim\langle \optb\rangle - \dim \geo{\optb}=\dim\langle \optb\rangle - \dim \grpy= \dim\varphi(\langle \optb\rangle)\le \dim\varphi(\langle \opta\rangle)
\end{equation}
where the fact that $\langle \optb\rangle$ contains a translate of $\grpy$
and basic dimension theory
are used 
for the second equality.

The torsion coset $\langle \opta\rangle$ is the translate of 
an abelian subvariety $\grpz$ of $\grpx$ by a torsion point. 
The fibres of $\varphi|_{\langle \opta\rangle}:\langle \opta\rangle\rightarrow
\varphi(\langle \opta\rangle)$ contain translates of $\grpy\cap \grpz$. 
Using dimension theory we find $\dim \varphi(\langle \opta\rangle) \le
\dim \langle \opta\rangle  - \dim \grpy\cap \grpz$. We observe
$\dim \grpy\cap \grpz \ge \dim \geo{\opta}$
and so $\dim \varphi(\langle \opta\rangle) \le \dim \langle \opta\rangle -
\dim\geo{\opta}$. Now let us combine this bound with
(\ref{eq:dimbound}) to deduce
\begin{equation*}
  \dim\langle \optb\rangle -  \dim\geo{\optb}\le 
  \dim\langle \opta\rangle -  \dim\geo{\opta}.
\end{equation*}

This inequality enables us to conclude
\begin{alignat*}1
  \defect(\optb) &= \dim \langle \optb\rangle-\dim \optb  \\
&\le \dim \langle \opta\rangle+\dim \geo{\optb} -\dim
\geo{\opta}-\dim \optb \\
&= \geodef(\optb)-\geodef(\opta)+\defect(\opta),
\end{alignat*}
as desired.

  In  case (iii), $\defect(B)-\geodef(B)$ is just the number of constant and 
non-special coordinates of $\langle B\rangle_{\rm geo}$. Then any 
weakly special subvariety containing $A$ (which is non-empty) but contained in 
$\langle B\rangle_{\rm geo}$, and in particular 
$\langle A\rangle_{\rm geo}$, has also at least that many
non-special constant coordinates.
\end{proof}


\begin{conjecture}
  Every mixed Shimura variety 
(and every weakly special subvariety) has the defect condition.
\end{conjecture}

\begin{proposition}
\label{prop:geodesicopt}
  Let $X$ have the defect condition, e.g.
$X$ is an abelian variety or $Y(1)^n$, and 
let $V\subset X$ be a subvariety.
An optimal subvariety of $V$ is geodesic-optimal in $V$.
\end{proposition}
\begin{proof}
Let $\opta\subset \var$ be an optimal subvariety and
  consider a subvariety $B$ with $A\subset B\subset V$ 
such that $\delta_{\rm geo}(B)\le \delta_{\rm geo}(A)$.
Then $\defect(B)-\geodef(B)\le \defect(A)-\geodef(A)$ and so
$$
\delta(B)=\delta_{\rm geo}(B)+\defect(B)-\geodef(B)\le 
\delta_{\rm geo}(A)+\defect(A) - \geodef(A).
$$
Since $A$ is assumed optimal we must have $B=A$,
and so $A$ is geodesic optimal.

Finally, the proposition applies to abelian varieties and $Y(1)^n$
because of Proposition \ref{prop:defectcondition}. 
\end{proof}

\section{Weak complex Ax}
\label{sec:wca}


In this section we formulate various Ax-Schanuel type Conjectures.
In the context of abelian varieties, these conjectures will be
theorems. But their modular counterparts are largely conjectural. 

As a warming-up let us recall a consequence of ``Ax-Schanuel'' [\AX] in the
complex setting.

Let now $A\neq \emptyset$ be an irreducible complex 
analytic subspace of some open $U\subset \CC^n$ such that
locally the coordinate functions $z_1,\ldots, z_n$ and
$\exp(z_1),\ldots, \exp(z_n)$ are defined and meromorphic 
on $A$.

\begin{definition}
  The functions $z_1,\ldots, z_n$ on $A$ 
are called  linearly independent over $\QQ$ mod $\CC$
if there are no non-trivial relations of the form
$\sum q_i z_i=c$ on $A$ where $q_i\in\QQ$ and $c\in\CC$.
\end{definition}

The linear independence of the $z_i$ over $\QQ$ mod $\CC$
on $A$
is equivalent to the multiplicative independence of the
$\exp(z_i)$ on $A$.

\begin{theorem}[Ax]
If the functions $z_1,\ldots, z_n$ on $A$ are
linearly independent over $\QQ$ mod $\CC$ then
$$
{\trdeg}_{\CC}\CC(z_1,\ldots, z_n, 
\exp(z_1),\ldots, \exp(z_n))\ge n+\dim A.
$$  
\end{theorem}

\subsection{The abelian setting}

We will state $2$ variations on Ax's Theorem that
are sufficient to treat unlikely intersections
in abelian varieties. 

Let $\grpx$ be an abelian variety defined over $\IC$.
We write $\ts{\grpx}$ for the tangent space of $\grpx$ at the
origin. 
Moreover, there is a complex analytic group homomorphism
$\exp:\ts{\grpx}\rightarrow \grpx(\IC)$. 
Here we use the symbol $\exp$ instead of $\pi$ to emphasise  the group
structure.

\begin{theorem}[Ax]
\label{thm:axabvar1}
  Let $U\subset \ts{\grpx}$ be a complex vector subspace and $z\in\ts{\grpx}$. 
Let $K$ be an irreducible analytic subset of an open neighborhood of $z$ in
$z+U$. 
If  $\optb$ is the Zariski closure of $\exp(K)$ in $\grpx$, then $\optb$ is
irreducible and
\begin{equation*}
  \geodef(\optb) \le \dim U - \dim K. 
\end{equation*}
\end{theorem}
\begin{proof}
  See Corollary 1 \cite{Ax72}.
\end{proof}

 The following
statement is sometimes called the Ax-Lindemann-Weierstrass Theorem for
abelian varieties.  Theorem \ref{thm:axabvar1} will be used 
in Section \ref{sec:finitegeoabvar}
whereas Ax-Lindemann-Weierstrass makes its appearance near the end in
Section \ref{sec:unlikelyabvar} where we apply it in connection with
a
 variant of the Pila-Wilkie Theorem.
Ax-Lindemann-Weierstrass is sufficient to prove
Theorem \ref{thm:curvesabvar}, but  Theorem \ref{thm:higherdimabvar},
situated in higher dimension, requires both statements. 
The reason seems to be that certain
technical difficulties disappear in    low dimension.

We may also consider $\ts{\grpx}$ as a real vector space of dimension
$2\dim \grpx$. After fixing an isomorphism
$\ts{\grpx}\cong \IR^{2\dim \grpx}$ it makes sense to speak about
semi-algebraic maps $[0,1]\rightarrow\ts{\grpx}$.

 \begin{theorem}[Ax]
\label{thm:ax2}
   Let $\beta:[0,1]\rightarrow \ts{\grpx}$ be real semi-algebraic and
   continuous  with $\beta|_{(0,1)}$ real analytic. 
The Zariski closure in $\grpx$ of the image of $\exp\circ \beta$
is a coset. 
 \end{theorem}
 \begin{proof}
Clearly, we may assume that $\beta$ is non-constant. 
The Zariski closure $\optb$ of the image $\exp(\beta([0,1]))$ is
irreducible since $\exp\circ\beta$ is continuous and real analytic on $(0,1)$.

By considering  Taylor expansions around points of $(0,1)$,
the restriction $\beta|_{(0,1)}$ extends to a holomorphic
map  $\gamma$ with target $\ts{\grpx}$ and defined
on a domain in $\IC$ which contains $(0,1)$.
By analyticity the image of $\exp\circ\gamma$ lies in $\optb$
and $\trdeg_\IC{\IC(\exp\circ\gamma)}\le \dim \optb$.

As $\beta$ is real algebraic on $[0,1]$ we find
 $\trdeg_\IC{\IC(\gamma)}\le 1$ and
therefore,
\begin{equation*}
  \trdeg_\IC{\IC(\gamma,\exp\circ\gamma)}\le 
\trdeg_\IC{\IC(\gamma)}+\trdeg_\IC{\IC(\exp\circ\gamma)}\le 1+\dim \optb. 
\end{equation*}

Let us apply the one variable case of Ax's Theorem 3 \cite{Ax72} 
 to the holomorphic function $t\mapsto \gamma(t+1/2) - \gamma(1/2)$
defined in a neighborhood of $0\in\IC$. 
According to the  inequality of
 transcendence degrees,
the smallest abelian subvariety $H\subset \grpx$
containing all values $\exp(\gamma(t+1/2)-\gamma(1/2))$
as $t$ runs over $U$ has dimension at most $\dim \optb$. Therefore,
$\optb = \exp(\gamma(1/2))+H$ which is what we claimed. 
 \end{proof}

\subsection{A product of modular curves}

Now we suppose that $X=Y(1)^n$ and that
$\pi:\HH^n\rightarrow X(\IC)$ is the
$n$-fold cartesian product of the $j$-function.

 Let again $A\neq\emptyset$ be an irreducible complex analytic 
subspace of some open $U\subset \HH^n$, so that
locally the coordinate functions $z_1,\ldots, z_n$ and
$j(z_1),\ldots, j(z_n)$ are defined and meromorphic 
on $A$, and we have a finite set $\{D_j\}$ of derivations
with ${\rm rank\/} (D_jz_i)=\dim A$, the rank being over the field
of meromorphic functions on $A$.

\begin{definition}
The functions $z_1,\ldots, z_n$ on $A$ are called
 geodesically independent if no $z_i$ is constant and there
are no relations $z_i=gz_j$ where $i\neq j$ and 
$g\in{\rm GL}_2^+(\QQ)$.
\end{definition}

The geodesic independence of the $z_i$ is equivalent 
to the $j(z_i)$ being ``modular independent'',
i.e. non-constant and no relations $\Phi_N(j(z_k), j(z_\ell))$ where
$k\ne \ell$ and $\Phi_N(X,Y)$ is a modular polynomial.

The following conjecture might be considered the analogue 
of ``Ax-Schanuel'' for the 
$j$-function in a complex setting.

\begin{conjecture}[Complex ``Modular Ax-Schanuel'']
In the above setting, if the $z_i$ are geodesically independent then
$$
{\trdeg}_{\CC}\CC(z_1,\ldots, z_n, j(z_1),\ldots, j(z_n))\ge n+\dim A.
$$  
\end{conjecture}

It evidently implies a weaker ``two-sorted'' version 
that, with the
same hypotheses, we have the weaker conclusion
$$
{\trdeg}_{\CC}\CC(z_1,\ldots, z_n)+
{\trdeg}_{\CC}\CC(j(z_1),\ldots, j(z_n))\ge n+\dim A.
$$
This conjecture is open beyond some special cases 
(``Ax-Lindemann'' [\PILAOAO] and ``Modular Ax-Logarithms''
[\HABEGGERPILA]). 

We pursue now some more geometric formulations.
To frame these we need some definitions. 

\begin{definition}
By a subvariety of $U\subset\HH^n$ we mean an irreducible component (in 
the complex analytic sense) of $W\cap U$ for some algebraic subvariety 
$W\subset \CC^n$.
\end{definition}

\begin{definition}
A subvariety $W\subset \HH^n$ is called  geodesic
if it is defined by some number of equations of the forms
$$
z_i=c_i,\quad c_i\in\CC;\quad z_k=g_{k\ell}z_{\ell},
\quad g\in{\rm GL}_2^+(\QQ).
$$  
\end{definition}

These are the ``weakly special subvarieties'' in the Shimura sense.

\begin{definition}
  By a  component we mean  a 
complex-analytically irreducible
component of $W\cap \pi^{-1}(V)$ where $W\subset U$
and $V\subset X$ are algebraic subvarieties.
\end{definition}

Let $A$ be a component of $W\cap\pi^{-1}(V)$. 
We can consider the coordinate functions
$z_i$ and their exponentials as elements of the field of
meromorphic functions (at least locally) on $A$, and we can endow
this field with $\dim A$ derivations in such a way that 
${\rm rank}(Dz_i)=\dim A$. Then (with ${\rm Zcl}$ indicating the
Zariski closure)
$$
\dim W\ge \dim {\rm Zcl\/} (A)
={\trdeg}_{\CC}\CC(z_1,\ldots,z_n),$$   
$$\dim V\ge \dim {\rm Zcl\/} (\pi (A)) ={\trdeg}_{\CC}\IC
(j\circ z_1,\ldots,j \circ z_n)
$$
and the ``two-sorted'' Modular Ax-Schanuel conclusion becomes
$$
\dim W+\dim V\ge \dim X+\dim A
$$
provided that the functions $z_i$ are
geodesically independent. 

\bigbreak

This condition is equivalent to $A$ being contained in a proper
geodesic subvariety. Let us take $U'$ to be the smallest geodesic
subvariety containing $A$. Let $X'= \exp U'$,
which is an algebraic subtorus of $X$, and put
$W'=W\cap U'$, $V'=V\cap X'$. We can choose coordinates
$z_i, i=1,\ldots,\dim A$ which are linearly independent over $\QQ$ mod $\CC$ 
and derivations such that ${\rm rank}(Dz_i)=\dim A$.
We get the following variant of Ax-Schanuel in this setting.


\begin{conjecture}[Weak Complex Ax (WCA): Formulation A.]
Let $U'$ be a geodesic subvariety of
$U$. Put $X'=\exp U'$ and let $A$ be a component of 
$W\cap \pi^{-1}(V)$, where $W\subset U'$ and $V\subset X'$
are algebraic subvarieties.
If $A$ is not contained in any proper geodesic subvariety of $U'$
then
$$
\dim A \le \dim V+\dim W-\dim X'.
$$
\end{conjecture}

I.e. (and as observed still more generally by Ax [\AY]), the intersections
of $W$ and $\pi^{-1}(V)$  never have ``atypically large''
dimension, except when $A$ is contained in a 
proper geodesic subvariety.
It is convenient to give an equivalent formulation.

\begin{definition}
  Fix a subvariety $V\subset X$.
  \begin{enumerate}
  \item [(i)]
A  component with respect to $V$ is a component of
$W\cap \pi^{-1}(V)$ for some algebraic subvariety 
$W\subset U$.
\item[(ii)]
 If $A$ is a component we define its  defect by
$\delta(A)=\dim {\rm Zcl\/}(A)-\dim A$.
\item[(iii)]
A component $A$ w.r.t. $V$ is called  optimal for $V$
if there is no strictly larger component $B$ w.r.t. $V$ with 
$\delta(B)\le \delta(A)$.
\item[(iv)]
A component $A$ w.r.t. $V$ is called  geodesic if it is a 
component of $W\cap \pi^{-1}(V)$ for some weakly special subvariety 
$W={\rm Zcl\/} (A)$.
  \end{enumerate}
 \end{definition}

\begin{conjecture}[WCA: Formulation B.]
 Let $V\subset X$ be a subvariety.
An optimal component with respect to $V$ is geodesic.
\end{conjecture}

Formulation B is the statement we need. However,
the two formulations are equivalent, and the proof of their
equivalence is purely formal and applies in the semiabelain setting
and indeed quite generally.

\begin{proof}[Proof that formulation A implies formulation B]
We assume Formulation A and 
suppose that the component $A$ of $W\cap \pi^{-1}(V)$ is optimal,
where $W={\rm Zcl\/} (A)$.
Suppose that $U'$ is the smallest geodesic subvariety containing $A$,
and let $X'=\pi(U')$.
Then $W\subset U'$. Let $V'=V\cap X'$. Then
$A$ is optimal for $V'$ in $U'$, otherwise it would fail to be
optimal for $V$ in $U$. Since $A$ is not contained in any proper geodesic subvariety of $U'$ we must have
$$
\dim A\le \dim W+\dim V'-\dim X'.
$$
Let $B$ be the component of $\pi^{-1}(V')$ containing $A$. Then
$B$ is also not contained in any proper geodesic subvariety of $U'$,
so, by Formulation A,
$$
\dim B\le \dim V'+\dim {\rm Zcl\/} (B) -\dim X'.
$$
But $\dim B=\dim V'$, whence $\dim {\rm Zcl\/} (B) = \dim X'$,
and so ${\rm Zcl\/} (B) = X'$, and $B$ is a geodesic component. Now
$$
\delta(A)=\dim W-\dim A\ge \dim X'-\dim V'=\delta(B)
$$
whence, by optimality, $A=B$.
\end{proof}

\begin{proof}[Proof that formulation B implies formulation A]
We assume Formulation $B$. Let $U'$ be a geodesic subvariety
of $U$, put $X'=\pi(U)$.
Suppose $V\subset X', W\subset U'$ are algebraic subvarieties and 
$A$ is a component of $W'\cap \pi^{-1}(V')$
not contained in any proper geodesic subvariety of $U'$.
There is some optimal component $B$ containing $A$, and $B$
is geodesic, but since $A$ is not contained in any proper geodesic,
$B$ must be a component of $\pi^{-1}(V')$ with
${\rm Zcl\/} (B)=U'$ and we have
$$
\dim W-\dim A\ge \delta(A)\ge \delta(B)=\dim X'-\dim V
$$
which rearranges to what we want.
\end{proof}

As already remarked, WCA holds for (semi-)abelian varieties,
by Ax [\AY] (see also [\KIRBY]).

To conclude we note that a true ``Modular Ax-Schanuel'' 
should take into account the derivatives of $j$.
A well-known theorem of Mahler [\MAHLER] implies that $j, j', j''$ are 
algebraically independent over $\CC$ as functions on $\HH$, and it is 
well-known too that $j'''\in\QQ(j, j', j'')$ (see e.g. [\BEZU]). 

\begin{conjecture}[Modular Ax-Schanuel with derivatives]
In the setting of ``Modular Ax-Schanuel'' above, if $z_i$
are geodesically independent then 
$$
{\trdeg}_{\CC}\CC(z_i, j(z_i) ,j'(z_i), j''(z_i)) \ge 
3n+\dim A.
$$  
\end{conjecture}

\section{A finiteness result for geodesic-optimal subvarieties}
\label{sec:finitenessgeodesic}


\subsection{The abelian case}
\label{sec:finitegeoabvar}

Suppose $\grpx$ is an
 abelian variety defined over $\IC$. 
In this section we prove the following finiteness statement on
geodesic-optimal subvarieties of a fixed subvariety of $\grpx$. 

\begin{proposition}
  \label{prop:geofinite}
For any subvariety $\var\subset \grpx$ there exists a
finite set 
 of abelian subvarieties of $\grpx$ with the
following property.
If $\opta$ is a geodesic-optimal subvariety of $\var$, then
 $\geo{\opta}$ is a translate of a member of the said set.
\end{proposition}

Any
positive dimensional, geodesic-optimal
subvariety $\opta\subsetneq \var$
 is ``$\mu$-anormal maximal'' for a certain $\mu$
in R\'emond's terminology  \cite{RemondInterIII}.
His Lemme 2.6 and Proposition 3.2 together imply that $\geo{\opta}$
is a translate of an abelian subvariety coming from a finite set that
depends only on $\var$. Thus our proof Proposition \ref{prop:geofinite} can be
regarded as an alternative approach to R\'emond's result using the language of
o-minimal structures. 

We retain the meaning of the symbol $\ts{\grpx}$ from Section \ref{sec:specialsubvar}
and  write $\exp:\ts{\grpx}\rightarrow \grpx(\IC)$  for the exponential map. It
is a holomorphic group homomorphism between complex manifolds 
whose kernel is the period lattice of $\grpx$.
We choose a basis of the period lattice and identify
$\ts{\grpx}$ with $\IR^{2g}$ as a real vector space. 
However, we continue to use both symbols $\ts{\grpx}$ and $\IR^{2g}$; the
former
is useful to emphasise the complex structure
and the latter is required as an ambient set for an o-minimal structure.


The open, semi-algebraic set $(-1,1)^{2g}$ contains a fundamental
domain for the period lattice $\IZ^{2g}\subset\IR^{2g}$ in real
coordinates. Under the identification $\IR^{2g}\cong \ts{\grpx}$ we fixed
above, we may consider $(-1,1)^{2g}$ as a domain in $\ts{\grpx}$. 

Let $\var$ be a subvariety of $\grpx$. Then 
\begin{equation*}
 \cV = \exp|_{(-1,1)^{2g}}^{-1}(\var(\IC)) 
\end{equation*}
is a subset of $\IR^{2g}$ and  definable in $\IRan$.
 But under the isomorphism  mentioned before, $\cV$
is also a complex analytic subset of $(-1,1)^{2g}\subset \ts{\grpx}$. Thus it is a complex
analytic space. The interplay of these  two points of view will have
many consequences. For an in-depth comparision between complex and o-minimal geometry
we refer to Peterzil and Starchenko's paper
\cite{PeterzilStarchenkoICM}.

Indeed, suppose $Z$ is an analytic subset of a domain in $\ts{\grpx}$ and
$z\in Z$. Then some open neighborhood of $z$ in $Z$ is definable in
$\IRan$
as $Z$ is defined by the vanishing of certain holomorphic functions. 
So the dimension $\dim_z Z$ of $Z$ at $z$ as a set definable in $\IRan$ is
well-defined.
But there is also the  notion of the dimension   of $Z$ at $z$ as a complex
analytic space \cite{CAS}. 

In this section we will add the subscript $\IC$ to the dimension
symbol to signify the dimension as a complex analytic space.

As $\dim\IC = 2 = 2\dim_{\IC}\IC$ the following lemma not
surprising. 
\begin{lemma}
  Let $Z$ be an analytic subset of a finite dimension
  $\IC$-vector space.
If $z\in Z$ then $\dim_z Z = 2\dim_{\IC,z} Z$.
\end{lemma}
\begin{proof}
Locally at $z$ the complex analytic space $Z$ is a finite union of
prime components, each of which is analytic in a neighborhood of $z$
in $Z$. Without loss
of generality we may assume that $Z$ is irreducible at $z$. After
shrinking $Z$ further we may suppose that $Z$ is irreducible,
definable in $\IRan$, and satisfies
$\dim_z Z = \dim Z$. We fix a decomposition of $Z$ into cells
$D_1\cup\cdots\cup D_N$ and write $Z'= Z \ssm\bigcup_{\dim D_i < \dim Z}
\overline{D_i}$ 
where the bar signifies closure in $Z$. Then $Z'$ is an open and non-empty subset of $Z$. 
So it must contain a smooth point $z'$ of the complex analytic space
$Z$. Around $z'$ we find $\dim_{z'} Z = 2\dim_{\IC,z'}  Z =
2\dim_{\IC,z}Z$ 
since $Z$, as an analytic space, is equidimensional. 
But $z'$ is contained in a cell $D_i$ of dimension $\dim Z$. So 
$\dim Z = 2\dim_{\IC,z}Z$. Our claim now follows from $\dim Z =
\dim_{z} Z$. 
\end{proof}

We fix an isomorphism of $\End{\ts{\grpx}}$, the endomorphisms of $\ts{\grpx}$
as a $\IC$-vector space, 
with $\IR^{2g^2}$ as an $\IR$-vector space.  
Let $\cO\subset \End{\ts{\grpx}}$  be definable in $\IRan$ and satisfy
 the
following $2$ properties. 
\begin{enumerate}
\item [(i)] We have $0\in \cO$. 
\item[(ii)] If $M\in \cO$ and if $\grpy\subset\grpx$ is an abelian
  subvariety of $\grpx$, then $\ts{\grpy}\cap\ker M$ is the kernel of
  an element in $\cO$.
\end{enumerate}
In particular, 
$\cO$ contains an element whose kernel is the tangent space of any
given abelian subvariety of $\grpx$. 

Suppose $0\le r\le g$ is  an integer. We set
\begin{alignat*}1
  \defF_r = \big\{&(z,M)\in \cV\times \cO;\,\,
\dim_\IC\ker M = r \text{ and }  
\text{for all } N\in \cO \text{ with } 
 \ker M\subsetneq \ker N: \\
 & \dim \ker M - \dim_z \cV\cap (z+\ker M) <
 \dim \ker N - \dim_z \cV\cap (z+\ker N) \big\}.
\end{alignat*}
Then  $\defF_r$ is definable in $\IRan$ by  standard properties of
o-minimal structures, for example by  Proposition 1.5, chapter 4
\cite{D:oMin}   taking local dimensions is
definable. 
We  set
\begin{alignat*}1
\defE_r = \big\{&(z,M)\in \defF_r;\,\, 
\text{for all }M'\in \cO \text{ with } 
 \ker M'\subsetneq \ker M: \\
&\dim_z \cV \cap(z+\ker M') < \dim_z \cV \cap(z+\ker M)\big\}
\end{alignat*}
which is also definable in $\IRan$.

 Ax's  Theorem  \ref{thm:axabvar1} for abelian varieties will be used in 
\begin{lemma}
\label{lem:axapp1}
  \begin{enumerate}
  \item [(i)]
If $(z,M)\in \defE_r$
there is an abelian subvariety $\grpy\subset \grpx$ with $\ts{\grpy} = \ker
M$. 
\item[(ii)] 
The set
\begin{equation*}
  \left\{ \ker M;\,\, (z,M)\in \defE_r \right\}
\end{equation*}
is finite. 
\item[(iii)] Let $\opta$ be a geodesic-optimal subvariety of $\var$
and let $\geo{\opta}$ be a translate of an abelian subvariety $\grpy$. 
If $r=\dim_\IC \grpy$ and
 $M\in \cO$ with $\ts{\grpy} = \ker M$
 then $(z,M)\in\defE_r$ for some $z\in (-1,1)^{2g}$.
  \end{enumerate}
\end{lemma}
\begin{proof}
  Say $(z,M)\in \defE_r$ is as in (i). We apply Ax's Theorem  \ref{thm:axabvar1}
 to $U=\ker M$. For this we  fix a prime component (in the complex
 analytic sense) $K\subset \cV \cap (z+U)$ with
$\dim_{\IC,z} K = \dim_{\IC,z} \cV\cap (z+U)$. 
By shrinking $K$ to an open neighborhood of $z$ we may assume that $K$
is irreducible and definable in $\IRan$.
Let $\optb\subset \grpx$ be as in Ax's theorem.

If $\geo{\optb}$ is a translate of the abelian subvariety
$\grpy\subset \grpx$ then $K\subset z+\ts{\grpy}$ and so
\begin{equation*}
  K \subset z + U\cap \ts{\grpy}. 
\end{equation*}
Observe that  $U\cap\ts{\grpy}$ is the kernel of an element in
$\cO$ by property (ii) above.
By the definition of $\defE_r$ we must have
$U\cap\ts{\grpy} = U$. Therefore, $U\subset\ts{\grpy}$. 

Next we prove that equality
holds. Indeed, if $U\subsetneq \ts{\grpy}$, then we may test $U$ against
any $N\in\cO$ with $\ker N =\ts{\grpy}$ in the definition of $\defF_r$. Therefore, 
\begin{alignat*}1
\dim U - \dim_z K &= 
  \dim U  - \dim_z\cV\cap (z+ U)\\
 &< \dim\ts{\grpy}  - \dim_z \cV\cap (z+\ts{\grpy})\\
&\le \dim\ts{\grpy}  - \dim_z \exp|_{(-1,1)^{2g}}^{-1}(\optb(\IC))
\end{alignat*}
where the final inequality required
$\optb\subset \var\cap (\exp(z)+\grpy)$. 
On passing to  complex dimensions we obtain
$\dim_\IC U - \dim_{\IC,z} K < \dim_\IC \grpy -\dim_{\IC,z} \optb =
\geodef(\optb)$ which
 contradicts the conclusion of Ax's Theorem. So we must have
$ \ker M = U = \ts{\grpy}$ and part (i) follows. 

Now we prove (ii) by showing that only finitely many  possible kernels $\ker M$ can arise
from  $(z,M)\in \defE_r$.
The image of $\defE_r$ under the projection
$\ts{\grpx}\times\End{\ts{\grpx}}\rightarrow\End{\ts{\grpx}}$
is definable in $\IRan$.
By Lemma \ref{lem:axapp1}(i) the kernel of $M$ is the tangent space of an abelian subvariety of
$\grpx$. 
But $\grpx$ has at most countably many abelian
subvarieties which leaves us with at most countably many possible kernels.
We fix a $\IC$-basis for $\ts{\grpx}$ and identify each $M$ with the
corresponding $g\times g$ matrix. The Pl\"ucker coordinates of a submatrix of $M$ with
maximal rank are in a countable set of an appropriate projective space. Pl\"ucker coordinates
are algebraic expressions in the entries of $M$. So we end up with a
countable and definable subset on each member of some affine covering
of projective space. 
A countable,
definable set is be finite and so there are at most finitely many $\ker M$.

Let $\opta$ and $\grpy$ be as in (iii). Since $\opta$ is a geodesic-optimal
subvariety of $\var$ it
must be an irreducible component of $\var\cap\geo{\opta}$. 
Let us fix $z\in \cV$ such that $\exp(z)$ is a smooth complex
point of $\opta$ that is not contained in any other irreducible component
of $\var\cap\geo{\opta}$. 

We first prove $(z,M)\in \defF_r$ by contradiction. 
So suppose there exists $N\in\cO$ with $\ts{\grpy}\subsetneq \ker N$ and
\begin{equation}
\label{eq:dimineq}
  \dim \ts{\grpy} - \dim_z \cV\cap (z+\ts{\grpy}) 
\ge \dim\ker N - \dim_z\cV\cap (z+\ker N). 
\end{equation}

As in (i) we  fix a prime component $K$
 of $\cV\cap(z+\ker N)$ that
passes through $z$ with $\dim_{\IC,z} K = \dim_{\IC,z}\cV\cap(z+\ker
N)$
and  is irreducible.
Let $\optb$ be the Zariski closure of $\exp(K)$, then Ax's Theorem
implies $\geodef(\optb)\le \dim_\IC \ker N - \dim_\IC K$. 
As $\exp$ is locally biholomorphic our choice of $z$ implies that $z$ is a smooth point
of the complex analytic set $\cV\cap (z+\ts{\grpy})$ which has dimension
$\dim_\IC \opta$ at this point. 
So
\begin{equation}
\label{eq:geodefineq}
  \geodef(\opta) = \dim_\IC \ts{\grpy} - \dim_\IC \opta \ge \dim_\IC\ker N -
  \dim_\IC K \ge \geodef(\optb) 
\end{equation}
follows from (\ref{eq:dimineq}) after dividing by $2$.

By smoothness, the intersection $\cV\cap(z+\ts{\grpy})$ has a unique prime component $K'$ passing
through $z$. 
The dimension inequality for intersections, cf. Chapter 5, \S 3 \cite{CAS}, implies
\begin{equation*}
  \dim_{\IC,z} K\cap (z+\ts{\grpy}) \ge
\dim_{\IC,z}K +\dim_\IC \ts{\grpy}- \dim_\IC\ker N . 
\end{equation*}
Inequality (\ref{eq:dimineq}) and the discussion above imply that the right-hand
side is at least
$\dim_\IC \opta$. 
But $K\cap(z+\ts{\grpy})\subset \cV\cap(z+\ts{\grpy})$ and on
comparing
dimensions at $z$  we find that $K\cap(z+\ts{Y})$, and a fortiori $K$,
contains a neighborhood of $z$ in $K'$. 
 This implies $\opta\subset \optb$. 
But (\ref{eq:geodefineq}) and the fact that $\opta$ is
a geodesic-optimal  subvariety
forces $\opta=\optb$. So $\dim_\IC K \le \dim_\IC \opta$ 
and (\ref{eq:geodefineq}) applied again yields
$\dim_\IC \ts{\grpy} \ge \dim_\IC \ker N$, a contradiction. 

Second, we will  show
$(z,M)\in\defE_r$. Suppose on the contrary that there is 
$M'\in \cO$ with $\ker M'\subsetneq\ts{\grpy}$ and
\begin{equation*}
  \dim_z \cV\cap(z+\ker M') = \dim_z \cV\cap(z+\ts{\grpy}). 
\end{equation*}
 The set on the right is  a complex analytic space, smooth at $z$, and contains
 the former.
So  $\cV\cap(z+\ker M')$ and $\cV\cap(z+\ts{\grpy})$ 
coincide on an open neighborhood of $z$ in $(-1,1)^{2g}$. 
Therefore, an open neighborhood of $0$ in $\opta(\IC)-\exp(z)$ is
contained
in the group $\exp(\ker M')$; here and below we use the Euclidean
topology. Said group need not be algebraic or even closed, but it
does contain an open, non-empty  subset of the complex points of
\begin{equation*}
  \underbrace{(\opta-\exp(z))+\cdots +(\opta-\exp(z))}_{\text{$\dim \grpx$ terms}}.
\end{equation*}
This sum  equals $\langle \opta-\exp(z)\rangle = \grpy  =
\exp(\ts{\grpy})$. Hence
$\ts{\grpy}\subset \ker M'$, which is the desired contradiction. 
\end{proof}


\begin{proof}[Proof of Proposition \ref{prop:geofinite}]
We will work with $\cO = \End{\ts{\grpx}}$. 
 Suppose that $\opta$ is a geodesic-optimal subvariety of $\var$.
Let us  fix $M\in\cO$
such that
$\geo{\opta}$
is the translate of an abelian subvariety whose tangent space is
$\ker M$. 
Then $\ker M$ 
lies in a finite set by  Lemma
\ref{lem:axapp1} parts (ii) and (iii). So $\geo{\opta}$ is the translate of
an abelian subvariety of $\grpx$ coming from a finite set.
\end{proof}

Let $\var\subset\grpx$ be a subvariety. 
We brief discussion the connection between Proposition \ref{prop:geofinite}
and anomalous subvarieties as introduced by Bombieri, Masser, and
Zannier
\cite{BMZGeometric}.

\begin{definition}
\label{def:anomalous}
Let $\grpx$ and $\var$ be as above. 
A 
 subvariety $\opta\subset \var$ is called anomalous if 
\begin{equation}
\label{eq:defanomalous}
 \dim \opta \ge \max\{1, \dim\geo{\opta}+\dim \var - \dim \grpx + 1\}. 
\end{equation}
If in addition 
$\opta$ is not contained in any strictly larger
anomalous subvariety of $\var$, then we call $\opta$ maximal
anomalous. 
The complement in $\var$ of the union of all anomalous subvarieties of
$\var$ is denoted by $\oa{\var}$. 
\end{definition}

Any maximal anomalous subvariety $\opta$ of $\var$ is
geodesic-optimal. Indeed, 
after enlargening there is a geodesic-optimal subvariety $\optb$ of $\var$ with
$\opta\subset \optb$ and $\geodef(\optb)\le \geodef(\opta)$. So
\begin{equation*}
  \dim \optb \ge \dim\geo{\optb} - \dim\geo{\opta} + \dim \opta
\ge \dim\geo{\optb} + \dim \var - \dim \grpx + 1
\end{equation*}
due to (\ref{eq:defanomalous}). Since $\dim \optb\ge \dim \opta\ge 1$ we see
that $\optb$ is anomalous. As $\opta$ is maximal anomalous we find
$\optb=\opta$. So $\opta$ is
geodesic-optimal.

According to  Proposition \ref{prop:geofinite}, $\geo{\opta}$ is the
translate of an abelian subvariety coming from a finite set which
depends only on $\var$. Let $\grpy$ be such an abelian subvariety. By a basic result in  dimension theory  of algebraic
varieties  the set of points $\var(\Kbar)$ at  which
$\grpx\rightarrow \grpx/\grpy$ restricted to $\var$ 
has a fibre greater or equal to  some prescribed value is Zariski closed
in $\var$. It follows that $\oa{\var}$ is Zariski open in $\var$;
we may thus use the notation $\oa{\var}(\Kbar)$ for $\Kbar$-rational
points in $\oa{\var}$. 

Openness of $\oa{\var}$ previously known due to work of R\'emond
\cite{RemondInterIII} and proved earlier in the toric setting by
Bombieri, Masser, and Zannier \cite{BMZGeometric}. 

We remark that $\oa{\var}$ is possibly empty. Moreover,
$\oa{\var}=\emptyset$ if and only if there exists a $\grpy$ as above such that
\begin{equation*}
  \dim \varphi(\var) < \min \{\dim \grpx/\grpy,\dim \var\} 
\end{equation*}
where $\varphi:\grpx\rightarrow \grpx/\grpy$.

\subsection{Mobius varieties}

Let $X=Y(1)^n$.

It is convenient to introduce a family subvarieties of $\HH^n$ 
parameterised by choices of elements of $\HH$ and 
${\rm GL}_2^+(\RR)$ 
in which weakly special subvarieties are the fibres corresponding
to the ${\rm GL}_2^+(\RR)$ parameters having their image
under scaling
in ${\rm SL}_2(\RR)$ lying in the image of
${\rm GL}_2^+(\QQ)$. Observe that this image is a countable set.
In [\PILAOAO] these were termed ``linear subvarieties'' but 
the denotation ``Mobius'' seems to be more appropriate.

We take $z_i$ as coordinates in $\HH^n$ and $g_i, i=2,\ldots, n$ as coordinates on 
${\rm GL}_2^+(\RR)^{n-1}$
The family of Mobius curves in $\HH^n$ is the locus
$$
M^{\{1,\ldots,n\}}\subset \HH^n\times {\rm GL}_2^+(\RR)^{n-1}
$$
defined by the equations $z_i=g_iz_1, i=2,\ldots, n$. We view this as a 
family of curves in $\HH^n$ parameterised by 
$g=(g_2,\ldots, g_n)\in  {\rm GL}_2^+(\RR)^{n-1}$.
For a subset 
$R\subset\{1,\ldots,n\}$ we define $M^R$ to be the family of 
Mobius curves
on the product of factors of $\HH^n$ over indices in $R$, parameterised
by the corresponding factors of ${\rm GL}_2^+(\RR)^{n}$ excluding the 
smallest one which plays the role of $z_1$, which we will denote 
${\rm GL}_2^+(\RR)^{R_i}$.

Let $R=(R_0, R_1,\ldots, R_k)$ be a strict partition as near
Definition \ref{def:wss}. 
Define the family
of Mobius subvarieties of type $R$ to be the locus
$$
M^R\subset \HH^n\times \HH^{R_0}\times 
\prod_{i=1}^k {\rm GL}_2^+(\RR)^{R_i}
$$
defined by equations placing the 
$\HH^{R_i}\times {\rm GL}_2^+(\RR)^{R_i}$
point in $M^{R_i}$ for $i=1,\ldots, k$, and each $R_0$-coordinate in $\HH^n$
is set equal to the corresponding coordinate in $\HH^{R_0}$. 
For each choice of parameters
$$
t\in M_R=\HH^{R_0}\times \prod_{i=1}^k {\rm GL}_2^+(\RR)^{R_i}
$$
the corresponding fibre $M^R_t$ is a {\it Mobius subvariety\/} $\HH^n$.

Like weakly special subvarieties, Mobius  subvarieties 
come in families of 
``translates''.
For a fixed $g\in \prod_{i=1}^k {\rm GL}_2^+(\RR)^{R_i}$, the choices of
$z\in \HH^{R_0}$ give a family of Mobius subvarieties, the ``translates'' 
of the corresponding ``strongly Mobius subvariety'' $M_g$ of the appropriate
subproduct of $\HH^n$, and the totality of the translates form a Mobius subvariety
with no fixed coordinates.

A component $A\subset \HH^n$ is contained in some smallest Mobius subvariety
$L_A$, 
has a Mobius defect 
$$
\delta_M(A)=\dim L_A-\dim A
$$
A component $A\subset j^{-1}(V)$ will be called 
{\it Mobius optimal\/} (for $V$ in $X$) if there is no component $B$ with 
$A\subset B\subset j^{-1}(V)$ and $\delta_M(B)\le \delta_M(A)$.

\begin{proposition}
 Assume WCA. Let $V\subset X$ be a subvariety. Then the set of 
$$
g\in\prod_{i=1}^k {\rm GL}_2^+(\RR)^{R_i}
$$
such that some translate of  $M_g$ intersects $j^{-1}(V)$  
in a component which is Mobius optimal for $V$ is finite
modulo the action by $\prod_{i}{\rm SL}_2(\ZZ)^{R_i}$. 
\end{proposition}
\begin{proof}
  The condition is 
$\prod_{i}{\rm SL}_2(\ZZ)^{R_i}$ invariant, so the
assertion is that such $g$ come in finitely many 
$\prod_{i}{\rm SL}_2(\ZZ)^{R_i}$
orbits.
By WCA, any such $g$ corresponds to a weakly special subvariety,
and so the $g$ in question belong to a countable set. 
However, every such $g$ has
a translate under $\prod_{i}{\rm SL}_2(\ZZ)^{R_i}$ 
for which the optimal component
has points of its full dimension in some fixed fundamental 
domain, say $\FF_0^n$, and there
the condition of optimality may be checked 
definably by considering dimensions
of the intersection of $\pi^{-1}(V)\cap \FF_0^n$ with 
Mobius subvarieties over the whole
space of them, which is definable. Thus, there is a definable 
(in $\RRanexp$)
countable and hence finite set of $g$ which contains 
a representative of every 
$\prod_{i}{\rm SL}_2(\ZZ)^{R_i}$ orbit of such $g$.
\end{proof}

As observed, the $g$ above all correspond to weakly special families; however,
every $g$ corresponding to a weakly special family having a translate with  a
geodesic-optimal intersection will also appear in this set, as such $g$
(by WCA) are in particular optimal among Mobius varieties. We conclude
a modular version of
Proposition \ref{prop:geofinite}.

\begin{proposition}
\label{prop:geofinitenessY1n}
 Assume WMA. Let $V\subset X$ be a subvariety. Then there is a finite 
set of basic special subvarieties such that every weakly special subvariety which 
has a geodesic-optimal component in its intersection with $V$ is a translate 
of one of these.
\end{proposition}

Finally, we observe that special subvarieties of type $R$
arise as fibres of $M^R$
of points of $M_R$ with suitable rationality properties.
Specifically, the coordinates in ${\rm GL}_2^+(\RR)^{R_i}$
are rational and those in $\HH^{R_0}$ are quadratic; let us call
these ``special points''. Of course the same fibre arises from 
non-special choices of the parameter too.

\begin{proposition}
   There is an absolute constant $c>0$
with the following property.
Let $T\subset Y(1)^n$ be a special subvariety with complexity
$\Delta(T)$ containing a point $P\in Y(1)^n$ with pre-image 
$Q\in \FF_0^n$.
Then there exists a ``special point'' $t\in M_R$ with 
$$
H(t)\le c\Delta(T)^{10}
$$
such that $M^R_t$ is a component of the pre-image of
$T$ and $Q\in M^R_t$.
\end{proposition}
\begin{proof}
  This follows from Lemmas  5.2 and 5.3 of 
[\HABEGGERPILA].
\end{proof}

\section{Counting semi-rational points}
\label{sec:semirational}


In this section we will work in a fixed o-minimal structure over
$\IR$. 
Our goal is to count points on a definable set where a certain
coordinates are algebraic of bounded height and degree
 and the rest are unrestricted. We will use our result to study
 unlikely intersections in abelian varieties.

Let us first fix some notation.
Let $k\ge 1$ be an integer.
We define the $k$-height of a real number $y\in \IR$ as
\begin{alignat*}1
  \Hpoly{k}{y} = \min \big\{\max\{|a_0|,\ldots,|a_k|\};\,\,
  &a_0,\ldots,a_k\text{ coprime integers, not all zero,}\\ 
&\text{with $a_0 y^k + \cdots + a_k = 0$ }\big\}
\end{alignat*}
using the convention  $\min\emptyset = +\infty$. 
A real number has finite $k$-height if and only if it has degree at
most $k$ over $\IQ$. 
Let $m\ge 0$ be an integer. For $y=(y_1,\ldots,y_m)\in\IR^m$ we set
\begin{equation*}
  \Hpoly{k}{y} = \max\{\Hpoly{k}{y_1},\ldots,\Hpoly{k}{y_m}\}
\end{equation*}
and abbreviate $\Hpoly{}{y}=\Hpoly{1}{y}$ if
$y\in\IR^m$.
If $\defZ\subset\IR^m$ is any subset, we define 
\begin{equation*}
  {\defZ}(k,T)= \left\{y\in \defZ;\,\, \Hpoly{k}{y}\le T\right\}
\end{equation*}
for $T\ge 1$.

If $\alpha$ is a mapping between two sets, then
$\Gamma(\alpha)$ will denote the graph of $\alpha$. 
Suppose $n\ge 0$ is an  integer.
If we state that a subset $\defF\subset \IR^m\times\IR^n$ is  a family  parametrised by
$\IR^m$, then $\defF_y$ stands for the projection of $\{y\}\times\IR^n\cap
\defF$ to $\IR^n$ if $y\in\IR^m$.

Let $\defZ\subset \IR^m\times\IR^n$ be a family parametrised by $\IR^m$. 
We want to determine the distribution of
 points $(y,z)\in \defZ$ where $y$ has $k$-height at 
 most $T$ without  restricting $z$.
For a real number $T\ge 1$, we define
\begin{equation*}
  \polyt{\defZ}{k,T}= \left\{(y,z)\in \defZ;\,\, \Hpoly{k}{y}\le T\right\}.
\end{equation*}
For technical reasons it is sometimes more convenient to work with 
\begin{equation*}
  \polytiso{\defZ}{k,T}= \left\{(y,z)\in   \polyt{\defZ}{k,T};
\,\, \text{$z$ is isolated in $\defZ_y$}\right\}.
\end{equation*}

Let $l\ge 0$ be an integer. 
We will use the second-named author's generalisation, stated
below, of the Pila-Wilkie Theorem  \cite{PilaWilkie}
 to prove the following result for definable families.

\begin{theorem}
\label{thm:relation}
 Let $\defF\subset \IR^l\times\IR^m\times\IR^n$ be a definable family
 parametrised by $\IR^l$ and  $\epsilon > 0$. 
There is a finite number $J=J(\defF,k,\epsilon)$ of block families
\begin{equation*}
\defW^{(j)}\subset \IR^{k_j}\times\IR^l\times\IR^{m},\quad j\in \{1,\ldots,J\}
\end{equation*}
parametrised by $\IR^{k_j}\times \IR^l$,
 for each such $j$ a continuous, definable function
\begin{equation*}
    \alpha^{(j)}:\defW^{(j)}\rightarrow \IR^n,
\end{equation*}
and a constant  $c=c(\defF,k,\epsilon)$ with the following properties.
\begin{enumerate}
\item [(i)] For all $j\in \{1,\ldots,J\}$ and all
 $(t,x)\in \IR^{k_j}\times\IR^{l}$ we have
    \begin{equation*}
\graph{\alpha^{(j)}}_{(t,x)} \subset \{(y,z)\in \defF_x;\,\, 
\text{$z$ is isolated in  $\defF_{(x,y)}$} \}.
\end{equation*}
\item[(ii)] Say $x\in \IR^{l}$ and $\defZ=\defF_x$. If $T\ge 1$ the set 
$\polytiso{\defZ}{k,T}$
 is contained in the union
of at most 
$c T^\epsilon$
graphs
$\graph{\alpha^{(j)}}_{(t,x)}$
  for suitable $j\in \{1,\ldots,J\}$
 and $t\in\IR^{k_j}$.
\end{enumerate}
\end{theorem}

What follows is a useful corollary of the result above. Its assertion deals with
$\polyt{\defZ}{k,T}$ and not $\polytiso{\defZ}{k,T}$
which appears in the theorem.

\begin{corollary}
\label{cor:relation}
Let $\defF$ and $\epsilon$ be as in
 Theorem \ref{thm:relation}.
We let $\pi_1$ and $\pi_2$ denote the projections
$\IR^m\times\IR^n\rightarrow\IR^m$ and
$\IR^m\times\IR^n\rightarrow\IR^n$, respectively. 
There exists a constant $c = c(\defF,k,\epsilon) > 0$
  with the following property. 
Say $x\in \IR^l$ and let $\defZ= \defF_x$.
If $T\ge 1$ and $\Sigma \subset \polyt{\defZ}{k,T}$ with 
  \begin{equation}
\label{eq:manypoints}
    \#\pi_2(\Sigma)  > cT^\epsilon,
  \end{equation}
 there exists a continuous and definable  function
$\beta:[0,1]\rightarrow \defZ$ such that the following properties hold.
\begin{enumerate}
\item [(i)] The composition $\pi_1\circ\beta:[0,1]\rightarrow\IR^m$ is
  semi-algebraic and its restriction to $(0,1)$ is real analytic.
\item [(ii)] The composition $\pi_2\circ\beta:[0,1]\rightarrow\IR^n$ is non-constant.
\item[(iii)]  We have $\pi_2(\beta(0)) \in \pi_2(\Sigma)$.
\item[(iv)]
If the o-minimal structure admits analytic cell
decomposition, then $\beta|_{(0,1)}$
is real analytic. 
\end{enumerate}
\end{corollary}


Here is the second-named author's counting theorem involving blocks. 

\begin{theorem}[Theorem 3.6 \cite{Pila:AO}]
\label{thm:pila}
 Let $\defF\subset \IR^l\times\IR^m$ be a definable family parametrised
 by $\IR^l$ and  $\epsilon > 0$. 
There is a finite number $J=J(\defF, \epsilon)$ of block families
$$
\defW^{(j)}\subset \IR^{k_j}\times\IR^l\times\IR^{m},\quad j=1,\ldots, J,
$$
each parametrised by $\IR^{k_j}\times\IR^l$,
and a constant  $c=c(\defF,k,\epsilon)$ with the following properties.
\begin{enumerate}
\item [(i)] For all $(t,x)\in \IR^{k_j}\times\IR^{l}$ and  all 
$j\in \{1,\ldots,J\}$ 
 we have $\defW_{(t,x)}\subset \defF_x$.
\item[(ii)] For all $x\in \IR^{l}$ and $T\ge 1$ the set 
$\poly{\defF_x}{k,T}$
 is contained in the union
of at most $cT^\epsilon$
blocks of the form $\defW^{(j)}_{(t,x)}$ for suitable $j=1,\ldots,
J$ and $t\in\IR^{k_j}$.
\end{enumerate}
\end{theorem}


\begin{proof}[Proof of Theorem \ref{thm:relation}]
We refer  to van den Dries's treatment of cells
 in Chapter 3  \cite{D:oMin}. 
His 
  convention for  a cell $\defC\subset \IR^m \times\IR^n$
has the following advantage when considering it as
 a family parametrised by $\IR^m$. If $y\in\IR^m$  then 
$\defC_y \subset \IR^n$ is either empty or a cell of dimension $\dim \defC
 - 1$.

We begin the proof of the theorem with a reduction step. Let us consider the set
\begin{equation*}
  \defF' = \{(x,y,z)\in \defF;\,\, \text{$z$ is isolated in $\defF_{(x,y)}$} \}.
\end{equation*}
We claim that it is definable. Indeed, let $\defC_1\cup\cdots\cup \defC_N$ be
a cell decomposition of $\defF$.
Then $\defF_{(x,y)} =
(\defC_1)_{(x,y)}\cup\cdots\cup (\defC_N)_{(x,y)}$ and each $(\defC_i)_{(x,y)}$ is
 either empty or a  cell.
The dimension of a non-empty
$(\defC_i)_{(x,y)}$ does not depend on $(x,y)$ and is the same locally at
all points. Therefore, $\defF'$ is a union
of a subclass of the $\defC_i$ and hence definable.  

Since $\defF'$ is definable it suffices to complete the proof with $\defF$
replaced by $\defF'$. We thus assume that $z$ is isolated in $\defF_{(x,y)}$
 for all
$(x,y,z)\in \defF$.

By general properties of an o-minimal structure,
the number of connected components in a
definable family is finite and bounded from above uniformly. 
So $\# \defF_{(x,y)}\le c_1$ for all $(x,y)\in \IR^l\times\IR^m$
where $c_1$ is independent of $x$ and $y$.

Let $\pi:\IR^l\times\IR^m\times\IR^n\rightarrow\IR^l \times\IR^m$ denote
the natural projection. Then $\defE^{(1)}=\pi(\defF)$ is a definable set. By Definable
Choice, Chapter 6 Proposition 1.2(i) \cite{D:oMin},
 there is a definable function $f^{(1)}:\defE^{(1)}\rightarrow \IR^n$ 
whose graph $\graph{f^{(1)}}$ lies in $\defF$. This graph is definable and
so is $\defF\ssm\graph{f^{(1)}}$. But the cardinality of a fibre of
$\defF\ssm\graph{f^{(1)}}$ considered as a family over $\IR^l\times\IR^m$ 
 is at most $c_1-1$. 

If $\defF\not= \Gamma(f_1)$, 
 Definable Choice yields a definable function
$f^{(2)}:\defE^{(2)}\rightarrow\IR^n$ on $\defE^{(2)} = \pi(\defF\ssm\graph{f_1})$
whose graph is inside $\defF \ssm \Gamma(f_1)$.  
The fibre above $(x,y)$ of $\defF\ssm (\Gamma(f_1)\cup \Gamma(f_2))$ has
at most $c_1-2$ elements.

This process
 exhausts  all fibres
of $\defF$ after  $c_2\le c_1$ steps. 
We get definable families
$\defE^{(1)},\ldots,\defE^{(c_2)}\subset\IR^l\times\IR^m$
parametrised by $\IR^l$ and definable functions
\begin{equation}
\label{eq:graphunion}
  f^{(i)} :\defE^{(i)}\rightarrow\IR^n \quad \text{for}\quad i\in \{1,\ldots,c_2\}
\quad\text{with}\quad
\bigcup_{i=1}^{c_2} \graph{f^{(i)}} =  \defF. 
\end{equation}

We can decompose each $\defE^{(i)}$  into finitely
many cells on which $f^{(i)}$ is continuous. So after possibly increasing
$c_2$ 
we may suppose that each $f^{(i)}$ is
continuous and definable. 

If $x\in\IR^l$, not all coordinates of a point in 
 $\polyt{\defF_x}{k,T}$ need to be  algebraic.
But the first $m$ coordinates are  and lead to 
  points of $k$-height  at most $T$ 
on some  $\defE^{(i)}_x$.
These points can be treated using Theorem
\ref{thm:pila} applied to the family $\defE^{(i)}$.
For every $i\in \{1,\ldots,c_2\}$ 
we  obtain $J_i$ block families $\defW^{(i,j)} \subset
\IR^{k_{i,j}}\times\IR^l\times\IR^m$ parametrised by $\IR^{k_{i,j}}\times\IR^l$ 
where $j\in\{1,\ldots, J_i\}$. They
satisfy
$\defW^{(i,j)}_{(t,x)}\subset \defE^{(i)}_x$
for $(t,x)\in \IR^{k_{i,j}}\times\IR^l$
and account for
all  points of  $k$-height at most $T$
on $\defE^{(i)}_x$. 

Note that if $(t,x,y)\in \defW^{(i,j)}$, then $(x,y)\in
\defE^{(i)}$. We consider the function
\begin{equation*}
\alpha^{(i,j)}:\defW^{(i,j)}\rightarrow \IR^n\quad\text{defined by}\quad
  (t,x,y) \mapsto f^{(i)}(x,y).
\end{equation*}
It is definable, being the composition of two definable
functions: a projection and $f^{(i)}$. 
Moreover, $\alpha^{(i,j)}$ is continuous by our choice of the $\defE^{(i)}$.
Observe $\Gamma(\alpha^{(i,j)})_{(t,x)}\subset\defF_x$ for all
$(t,x)\in\IR^{k_{i,j}}\times\IR^l$
and this will yield (i). 

Suppose $x\in\IR^l$ and $(y,z)\in \polyt{\defZ}{k,T}$ with
$\defZ = \defF_x$. 
The point $(x,y,z)\in \defF$   lies  on the graph of some $f^{(i)}$
by (\ref{eq:graphunion}). 
Hence $z = f^{(i)}(x,y)$ with $(x,y)\in \defE^{(i)}$. 
By definition we have 
$y\in \defE^{(i)}_x$. So $y \in \defE^{(i)}_x(k,T)$ since $y$ has $k$-height at
most $T$. 
Suppose $c(\defE^{(i)},k,\epsilon)$ is the constant from   Theorem \ref{thm:pila}. Then
 $y$ is inside some
 $\defW^{(i,j)}_{(t,x)}$ where $j$ and $t$ are allowed to vary
over $c(\defE^{(i)},k,\epsilon) T^\epsilon$ possibilities. 
Therefore, $(t,x,y)\in \defW^{(i,j)}$ and
$(t,x,y,z) \in \Gamma(\alpha^{(i,j)})$
or equivalently, $(y,z) \in \Gamma(\alpha^{(i,j)})_{(t,x)}$. 
Part (ii) and the theorem follow after renumbering the $\alpha^{(i,j)}$ and
$\defW^{(i,j)}$. 
\end{proof}

\begin{proof}[Proof of Corollary \ref{cor:relation}]
  The constant $c$ from the this corollary comes from Theorem
  \ref{thm:relation} applied to $\defF,k,$ and $\epsilon$.

Let $x\in \IR^l$ and $\defZ= \defF_x\subset\IR^m\times\IR^n$. 
Suppose $T\ge 1$ satisfies
\begin{equation}
\label{eq:toomanypoints}
  \# \pi_2(\Sigma) > c T^{\epsilon}
\end{equation}
with $\Sigma\subset \polyt{\defZ}{k,T}$ as in the hypothesis. 

First, let us suppose $\Sigma\not\subset\polytiso{\defZ}{k,T}$. 
We fix any $(y,z) \in \Sigma\ssm \polytiso{\defZ}{k,T}$; then
 $\Hpoly{k}{y}\le T$
and the connected component of 
 $\defZ_y$ containing $z$ has positive dimension.
This component is definably connected. 
So we may fix a definable and continuous path
 $\alpha:[0,1]\rightarrow \defZ_y$  connecting $\alpha(0)=z$ with any 
other point 
$\alpha(1)\not=z$ of said component.  Properties (i)-(iii) follows with 
the function $\beta(t) = (y,\alpha(t)) \in \defZ$  for $t\in [0,1]$. 
Rescaling  implies (iv).

We have reduced to the case $\Sigma\subset \polytiso{\defZ}{k,T}$.
By our Theorem \ref{thm:relation} the set
$\Sigma$
is contained in the union of
at most $c T^\epsilon$ graphs of  continuous and 
definable  functions. 
The Pigeonhole Principle and (\ref{eq:toomanypoints}) yield two
$(y,z), (y',z') \in \Sigma$ on the same graph
with
\begin{equation}
\label{eq:differentz}
  z = \pi_2(y,z)\not=\pi_2(y',z')=z'.
\end{equation}
So there is a  block family $\defW\subset
\IR^k\times\IR^l\times \IR^m$ and a continuous, definable function
$\alpha:\defW\rightarrow\IR^n$
with $(y,z),(y',z')\in \Gamma(\alpha)_{(t,x)}$ for a certain
$t\in \IR^k$. 
Moreover, 
$\Gamma(\alpha)_{(t,x)} \subset \defZ$. 

The fibre $\defW_{(t,x)}$ is  a block containing $y$ and $y'$. 
A block is connected by definition. As above this 
 means that
 there is a continuous, definable  function
$\gamma :[0,1]\rightarrow \defW_{(t,x)}$ with
$\gamma(0) = y$ and $ \gamma(1) = y'$. 
But $\defW_{(t,x)}$, being a block, is locally a semi-algebraic
set. 
That is, for any $s\in [0,1]$ the point    $\gamma(s)$ has a semi-algebraic
neighborhood in $\defW_{(t,x)}$.
Because $[0,1]$ is compact we may assume that
$\gamma$ is semi-algebraic.
 
We set
\begin{equation*}
  \beta(s) = (\gamma(s), \alpha(t,x,\gamma(s))) 
\end{equation*}
and this yields a  function
$\beta : [0,1]\rightarrow \defZ$ which we show to satisfy all points
in the assertion. 

The function $\beta$ is 
 continuous and definable
 as $\gamma$ and $\alpha$ possess these properties.

We note that $\pi_1\circ\beta = \gamma$  is
semi-algebraic by construction. This yields the first statement in (i).

We also note $\alpha(t,x,y)=z$, so $\beta(0) = (y,z) \in \Sigma$ 
and (iii) follows.  

We find $\beta(1) = (y',z')$ in a analog manner.
Therefore,
(\ref{eq:differentz}) implies $\pi_2(\beta(0)) \not = \pi_2(\beta(1))$
and (ii) follows from this.

To complete the proof of (i) we use the fact that $\IRalg$
admits analytic cell decomposition.
There exist
$0=a_0 < a_1 < \cdots <a_{k+1}=1$ such that each
$\pi_1\circ \beta|_{(a_i,a_{i+1})} :(a_i,a_{i+1})\rightarrow \IR^{m}\times\IR^n$
is real analytic.
By continuity and (ii) the restriction of $\pi_2\circ\beta$ to some 
interval $(a_{i},a_{i+1})$ is non-constant.
If $i$ is minimal with this property, then 
$\pi_2(\beta(a_i))=\pi_2(\beta(0)) = z$. 
This will preserve (iii)
after a linear reparametrisation
of $(a_i,a_{i+1})$ to  $(0,1)$.
Thus we may suppose that $\pi_1\circ\beta|_{(0,1)}$ is real analytic
and this completes (i).

To prove (iv) we must assume that the ambient o-minimal structure 
admits analytic cell decomposition. As before
we cover
$[0,1]$ by finitely many open intervals and points,
such that  $\pi_2\circ\beta$
restricted to each open interval is real analytic. 
We again linearly 
rescale  the first  open interval on which
$\pi_2\circ\beta$ is non-constant to $(0,1)$. So 
 $\beta_{(0,1)}:(0,1)\rightarrow\IR^m\times\IR^n$ is
real analytic. 
\end{proof}

Say $\defZ \subset \IR^m$ is a definable. The corollary applied (with $m=n$) to the
diagonal embedding $\defF = \{(z,z);\,\, z\in \defZ \}\subset \IR^m\times\IR^m$ recovers 
Pila and Wilkie's Theorem 1.8 \cite{PilaWilkie}.

\section{Large Galois orbits}
\label{sec:lgo}


Let $X$ be $Y(1)^n$ or an abelian variety defined over a  field
$K$ which we take to be finitely generated over $\QQ$.
Recall that we have a notion of complexity of special subvarieties of
$X$, cf. Sections \ref{sec:degabsv} and \ref{sec:specialsubvarY1n}
for the abelian and modular cases, respectively. 
Suppose 
 $V\subset X$ is a subvariety defined over $K$.
We consider various assertions. 

\begin{enumerate}
\item [(GO1)]For $V \subset X$ there exist $c,\eta >0$ 
such that, for all $P \in V$,
$$
[K(P):K]\ge c\Delta(\langle P\rangle)^\eta.
$$
\item[(GO2)]
As in (GO1)  but only when the component $\{P\}$ of 
$V \cap \langle P\rangle$ is optimal.
\item[(GO3)]
For $V \subset X$ there exist $c,\eta >0$ such that, 
for all components $W$ of $V \cap \langle W\rangle$,
$$
[K(W):K]\ge c\Delta(\langle W\rangle)^\eta.
$$
\end{enumerate}

Now GO3 is simply  too strong: e.g. if $V =X=Y(1)^n$ then a 
component of $V \cap T$ is $T$ and $\langle T\rangle = T$ 
for any special subvariety, but some specials have large 
complexity and  small  Galois orbits 
(e.g. any strongly special $T \subset Y(1)^n$).

Also GO1 seems to be very strong. Taking $V = X = Y(1)^n$ 
and  $P \in\QQ^n$, the existence 
of $c,\delta$ means that $\Delta(\langle P\rangle )$ is 
bounded as $P$  runs over all rational points,  
so that only finitely many  $\langle P\rangle$ arise. 
But this could be true.
E.g. in  $Y(1)^2$ we know that there are only finitely many modular 
curves with non-CM, non-cuspidal rational points (Mazur). 
It is however stronger than we need.

For $V = Y(1)^n$ we see that $P$ is a component 
if and only if it is special, and so the statement of GO2 in that case reverts to the 
conjecture on Galois orbits of  special points. Also GO1 is odd if $P$ is 
transcendental, while in GO2 it must be algebraic over $K$.
For positive dimensional components we must depend on o-minimality 
(and WCA) to bring us to finitely many families, 
and reduce to considering their  translates.


\begin{definition}
\label{def:LGO}
Let $K$ be a field that is finitely generated over $\IQ$. 
Suppose $X$ is $Y(1)^n$ or an abelian variety defined over
$K$ and
 $\var\subset \grpx$ is a subvariety also defined over
$K$. Let $s\ge 0$.
We say that $\lgo_s(\var)$ is satisfied 
if there exists a constant $\kappa > 0$ 
with the following property. 
For any $P\in \var(\overline K)$ such that
$\{P\}$ is an optimal singleton of $\var$ with $\dim \langle P\rangle
\le s$ we have
\begin{equation}
\label{eq:deglb}
\Delta(\langle P\rangle) \le (2[K(P):K])^\kappa.
\end{equation}

If $r\ge 0$, we say that $\grpx$ satisfies $\lgo_s^r$, if
$\lgo_s(\var)$ is satisfied for all 
$\var\subset \grpx$ defined over $\Kbar$ above with $\dim \var\le r$.   

Finally, we say that $\grpx$ satisfies $\lgo$ if it satisfies
$\lgo_s^r$ for all $r,s\ge 0$. 
\end{definition}

\begin{conjecture}
Let $K$ be finitely generated over $\IQ$.
If $X$ is an abelian variety defined over $K$ or if  $X=Y(1)^n$, then 
$X$ satisfies $\lgo$. 
\end{conjecture}

One might expect an analog conjecture to hold in any mixed 
Shimura variety (or perhaps even any weakly special subvariety
thereof).

Suppose $K$ is a number field. 
For special points of Shimura varieties the best results known are
those of Tsimerman [\TSIMERMAN]: lower bounds 
of the above form for the size of the Galois orbit for special points 
in the coarse moduli space of Abelian varieties of dimension $g$, 
${\cal A}_g$, for $g\le 6$ (or on GRH for all $g$; see also 
[\ULLMOYAFAEV]). For unlikely
intersections of curves with special subvarieties of $Y(1)^n$ 
partial results are obtained in [\HABEGGERPILA].

To prove a uniform version (which one could frame as a ZP for 
$Y(1)^n\times \CC^k$
where the second factor is viewed with rational structure, as done in 
[\PILAEMS]), one would want that, for a family of subvarieties 
$\{V_t\}$, i.e. the fibres of some $V\subset Y(1)^n\times \CC^k$,
the constant $\kappa$ for $V_t$ in the family depends only on 
$[\QQ(V_t):\QQ]$, 
i.e. on the degree of the parameter $t$ over $\QQ$.

\section{Unlikely intersections in abelian varieties}
\label{sec:unlikelyabvar}


\subsection{The arithmetic complexity of a torsion coset}

In this section we will prove an upper bound on the arithmetic complexity  of a torsion coset
as introduced in  Definition \ref{def:complexitytorsioncoset}.

Suppose $\grpx$ is an
 abelian variety of dimension $g$ defined over a number field $K$
and $\cL$ is an ample line bundle on $\grpx$. 
To simplify notation we will suppose that $K$ is a subfield of $\IC$.
For example, this enables us to 
consider the tangent space $\ts{\grpx}$ as a $\IC$-vector space.
We let $\overline K$ denote the algebraic closure of $K$ in $\IC$.

After replacing $\cL$ by $\cL\otimes\cL^{\otimes (-1)}$ we may
 assume that $\cL$ is symmetric. Let $\hat h$ be the
N\'eron-Tate height on $\grpx(\overline K)$ attached to $\cL$. 
We recall that the group of homomorphisms between two abelian
varieties is a  finitely generated, free abelian group. 
Let $d$ be the dimension of an abelian subvariety of $\grpx$. 
For the next proposition we require  
\begin{alignat*}1
  \lambda_\grpx(d) = \sup\{ \dim (\grpx/H) \cdot\rank \Hom{\grpx,\grpx/H};\,\, &H\subset
  \grpx\text{ is an abelian subvariety} \\
&\text{over $\overline K$ with } \dim H=d\}  < +\infty 
\end{alignat*}
where $\rank$ denotes the rank of a free abelian group and
$\Hom{\cdot,\cdot}$
the group of homomorphisms over $\overline K$. 

We observe that $\lambda_\grpx(d)=0$ if and only if $d=\dim \grpx$.

\begin{proposition}
  \label{prop:boundcomplexity}
There exists a constant $c>0$ depending on $\grpx$ and $\cL$ such 
that
\begin{equation*}
  \Delta_{\arith}(\langle P\rangle) \le c [K(P):K]^{6g+1}
\end{equation*}
and
\begin{equation*}
  \deg_\cL H \le c [K(P):K]^{60g^4}\max\{1,\hat
   h(P)\}^{\lambda_\grpx(\dim \langle P\rangle)}
\end{equation*}
for all
 $P\in \grpx(\Kbar)$
where  $H=\langle P\rangle-P$. In particular,
\begin{equation*}
  \Delta(\langle P\rangle) \le
c [K(P):K]^{60g^4}\max\{1,\hat
   h(P)\}^{\lambda_\grpx(\dim \langle P\rangle)}.
\end{equation*}
\end{proposition}

Bombieri, Masser, and Zannier \cite{BMZ} already employed essentially 
best-possible height lower bound due to Amoroso and David \cite{ADLehmerSup} to prove a
weak form of the Zilber-Pink Conjecture  for curves in $\IG_m^n$. 
Later, R\'emond \cite{RemondInterI} developed this approach 
for abelian varieties using the geometry of numbers. 
We will follow a similar line of thought in this section. 
However, lower bounds of the same quality as Amoroso and David's
result are not known for a general abelian variety. 
 One advantage of the  o-minimal
approach is that it can cope with a  height lower
bound as long as it is polynomial in the reciprocal of the degree.
Below, we will use such a lower bound due to Masser
\cite{Masser:smallvalues}
that applies to any abelian variety defined over a number field.
For the sake of simplicity we state the estimates in a form that is
weaker than what Masser proved.  

\begin{theorem}[Masser]
\label{thm:masser}
There exists a constant $c>0$ depending only on $\grpx,K,$ and $\cL$ 
with the following property.
Suppose $P\in \grpx(\Kbar)$ and $D=[K(P):K]$ 
with $\hat h(P) < c^{-1} D^{-2g -
  9}$, then $P$ is a torsion point of order
at most $c D^{6g + 1}$. 
\end{theorem}
\begin{proof}
  This  follows from the main theorem of \cite{Masser:smallvalues} and
  the comments that follow  it.    
\end{proof}
Given $P\in \grpx(\Kbar)$, a first consequence of Masser's Theorem is a lower bound for the
$[K(P):K]$
in terms of the arithmetic complexity of $\langle P\rangle$.
We do not yet need the N\'eron-Tate height. 

\begin{lemma}
\label{lem:arithlb}
There exists a constant $c_3>0$ such 
that $\Delta_{\arith}(\langle P\rangle) \le c_3[K(P):K]^{6g+1}$
for all
 $P\in \grpx(\Kbar)$.
\end{lemma}
\begin{proof}
We note that the conclusion becomes stronger when replacing $K$ by a field
extension. So 
 we may assume that all abelian subvarieties of $\grpx$ are
defined over $K$.

   Bertrand proved that there is an integer $c_1\ge 1$ such that any
  abelian subvariety $\grpy\subset \grpx$ has a companion abelian
  subvariety
$\grpz\subset \grpx$ with
$\grpy+\grpz=\grpx$ such that $\grpy \cap \grpz$ is finite and 
contains at most $c_1$ elements.
Ratazzi and Ullmo published a proof \cite{RatazziUllmo} of
Bertrand's Theorem.
The point here is that $c_1$ does not depend on $\grpy$.

Suppose that $\langle P\rangle$ is a translate of $\grpy$
by a torsion point.
Let us write $P = Q+R$ with $Q\in \grpy(\Kbar)$ and $R\in \grpz(\Kbar)$. 
A positive multiple of $P$ lies in $\grpy(\Kbar)$. This
and $\#\grpy\cap \grpz< \infty$
imply that $R$ has finite order, say $M$.
Masser's Theorem implies $M\le c_2 [K(R):K]^{6g+1}$.

By definition we have $\Delta_{\arith}(\langle P\rangle) \le M$
and thus
 $\Delta_{\arith}(\langle P\rangle)\le c_2 [K(R):K]^{6g+1}$.
It remains to bound $[K(R):K]$ from above in terms
of $[K(P):K]$. 

Suppose $\sigma,\sigma'\in{\rm Gal}(\Kbar/K)$ with
$\sigma(R)\not=\sigma'(R)$.
If $\sigma(P)=\sigma'(P)$ then $P=Q+R$ yields
$\sigma'(Q)-\sigma(Q)=\sigma(R)-\sigma'(R)$.
This point is in $\grpy\cap \grpz$
 as $\grpy$ and $\grpz$ are defined over $K$.
This leaves at most $c_1$ possibilities for $\sigma(R)-\sigma'(R)$. We conclude
\begin{alignat*}1
[K(P):K] &= \# \{\sigma(P);\,\, \sigma\in{\rm Gal}(\Kbar/K) \}\\
&\ge\frac{1}{{c_1}} \# \{\sigma(R);\,\, \sigma\in{\rm Gal}(\Kbar/K) \}\\
&= \frac{[K(R):K]}{{c_1}}.
\end{alignat*}
The lemma follows from the lower bound for $[K(R):K]$. 
\end{proof}

We setup some additional notation before we come to the proof of Proposition \ref{prop:boundcomplexity}. 

Say $\grpy$ is a second abelian variety, also defined over $K$,
equipped with an ample and symmetric line bundle. Thus we  
have another  N\'eron-Tate
height
$\hat h_\grpy:\grpy(\overline K)\rightarrow [0,\infty)$. 
We will assume that all elements in $\Hom{\grpx,\grpy}$ are already
defined over $K$.
We set $\rho=\rank \Hom{\grpx,\grpy}$.
To avoid trivialities we shall assume $\rho \ge 1$, so in particular
$\dim Y \ge 1$. 
We set
\begin{equation*}
  \hommaxR{\grpx,\grpy}  = \{\psi \in \Hom{\grpx,\grpy}\otimes \IR;\,\,
\text{there is }\varphi\in \Hom{\grpy,\grpx}\otimes \IR \text{ with
}\psi\varphi = 1\}
\end{equation*}
and fix a norm $\|\cdot\|$ on $\Hom{\grpx,\grpy}\otimes\IR$.  For
example, we could take the norm induced by the
Rosati involution coming from  the two line bundles.
We consider $\Hom{\grpx,\grpy}\otimes\IR$
with the induced topology.
An element $\psi\in\Hom{\grpx,\grpy}\otimes\IR$ lies in
  $\hommaxR{\grpx,\grpy}$  precisely when the linear
map $\varphi\mapsto \psi \varphi$ is surjective.
Therefore, $\hommaxR{\grpx,\grpy}$ is an open, possibly empty, subset
of $\Hom{\grpx,\grpy}\otimes\IR$.

In this section,  $c_1,c_2,\ldots$ denote positive constants that depend only on
these two abelian varieties, $K$,  and the choosen line bundles.

The upper bound for the
 geometric part of the complexity involves the N\'eron-Tate height.


\begin{lemma}
\label{lem:findvarphi}
 Suppose  $Q > 1$ and  
let $P\in \grpx(\Kbar)$ be in the kernel of a surjective element of $\Hom{\grpx,\grpy}$.
 There is a surjective $\varphi\in\Hom{\grpx,\grpy}$ with
\begin{equation*}
\hat h_\grpy(\varphi(P)) \le c_4 Q^{-2/\rho}\hat h(P)
\quad\text{and}\quad 
\|\varphi\| \le c_4 Q. 
\end{equation*}
\end{lemma}
\begin{proof}
By Lemmas 2 and 5 \cite{abvar} 
there is a compact subspace
$\cK\subset\hommaxR{\grpx,\grpy},\varphi\in\Hom{\grpx,\grpy},
\varphi_0\in\cK,$
and an integer $q$ with  $1\le q\le Q$ 
such that $\hat h_\grpy(\varphi(P)) \le c_4 Q^{-2/\rho}\hat h(P)$
and $\| q\varphi_0 - \varphi\| \le c_4 Q^{-1/\rho}$. 
We emphasise that $c_4$ does not depend on $P$ or $Q$.

The norm is bounded from above on the compact space $\cK$. So we obtain $\|\varphi\| \le
\|q\varphi_0-\varphi\|+q\|\varphi_0\| \le c_4 Q$ after increasing
$c_4$, if necessary. 

We must show that $\varphi$ is surjective. Recall that $\hommaxR{\grpx,\grpy}$ is open.
 Therefore, there is $Q_0\ge 1$ such that  $Q>Q_0$ 
and
 $\|\varphi_0-\varphi/q\| \le c_4 Q^{-1/\rho}$ imply
$\varphi/q\in \hommaxR{\grpx,\grpy}$. In particular
$\varphi$ has a right inverse in $\Hom{\grpy,\grpx}\otimes\IR$. 
In this case it must already have a right inverse in
$\Hom{\grpy,\grpx}\otimes\IQ$ by basic linear algebra. So $\varphi$ is
surjective. 

Now if $Q\le Q_0$ it suffices to take a fix surjective homomorphism
for $\varphi$. By general properties of the N\'eron-Tate height
 $\hat h_\grpy(\varphi(P))$ is bounded from above linearly in $\hat
h(P)$, cf. expression (8) in Section 2 \cite{abvar}. The
lemma follows after increasing $c_4$ a final time. 
\end{proof}

\begin{lemma}
\label{lem:galoisorbit}
Suppose
 $P\in \grpx(\Kbar)$ is  in the kernel of a surjective
element of $\Hom{\grpx,\grpy}$ and
  $D=[K(P):K]$. There exists a surjective $\varphi\in
  \Hom{\grpx,\grpy}$ with $\varphi(P)=0$ and
  \begin{equation*}
    \|\varphi\| \le c_6 D^{6\dim \grpy+1+(2\dim \grpy+9)\rho/2} \max\{1,\hat h(P)\}^{\rho/2}.
  \end{equation*}
\end{lemma}
\begin{proof}
  Suppose $P$ is as in the hypothesis and let us abbreviate
$h = \max\{1,\hat h(P)\}$. We suppose that $\varphi$ is a
  surjective morphism with $\varphi(P)=0$ and with 
 $\|\varphi\|$   minimal among all such morphisms. 
Let $c_5>0$ be the constant from Masser's Theorem applied to $\grpy$. 
We will assume 
  \begin{equation}
\label{eq:varphilb}
    \|\varphi\| > 2(c_4c_5)^{1+\rho/2} D^{6\dim \grpy+1+(2\dim \grpy+9)\rho/2}h^{\rho/2},
  \end{equation}
with $c_4$ from Lemma \ref{lem:findvarphi},
and derive a contradiction. This implies the proposition with $c_6 =
2(c_4c_5)^{1+\rho/2}$. 

Let us define the integer
\begin{equation*}
  N = [c_5 D^{6\dim \grpy+1}].
\end{equation*}
Without loss of generality we have $c_5\ge 1$,
so $N\ge 1$. We
define further
\begin{equation*}
  Q = \frac{\|\varphi\|}{2c_4N}.
\end{equation*}
Our assumption implies $\|\varphi\| > 2c_4c_5 D^{6\dim \grpy+1} \ge 2c_4N$ and
so $Q>1$. We apply Lemma \ref{lem:findvarphi} to find a surjective  $\phi \in \Hom{\grpx,\grpy}$
 with $\hat h_\grpy(\phi(P)) \le c_4 Q^{-2/\rho} h$
 and $\|\phi\| \le c_4 Q$.
  
Say $1\le n\le N$, then 
\begin{equation*}
  \|n\phi\| \le N\|\phi\| \le c_4 NQ = \frac{\|\varphi\|}2 < \|\varphi\|.
\end{equation*}
By minimality of $\|\varphi\|$ we conclude $n\phi(P)\not=0$. So $\phi(P)$ is either
non-torsion or has finite order strictly greater than $N$.  Masser's
Theorem
 excludes the second alternative and 
provides
$\hat h_\grpy(\phi(P)) \ge c_5^{-1} D^{-2\dim \grpy-9}$. We combine this bound with the
upper bound from Lemma \ref{lem:findvarphi} to deduce
$c_4c_5 D^{2\dim \grpy+9}h \ge Q^{2/\rho}$. Inserting our choice for $Q$ and $N$ gives
\begin{equation*}
  c_4 c_5  D^{2\dim \grpy+9} h\ge \left(
\frac{\|\varphi\|}{2c_4N}\right)^{2/\rho}
\ge \left(
\frac{\|\varphi\|}{2c_4c_5 D^{6\dim \grpy+1}}\right)^{2/\rho}.
\end{equation*}
The  incompatibility with (\ref{eq:varphilb}) is the
desired contradiction. 
\end{proof}

\begin{proof}[Proof of Proposition \ref{prop:boundcomplexity}]

The bound for the arithmetic part of the complexity follows from Lemma
\ref{lem:arithlb}. The complexity of 
$\langle P\rangle$ is the maximum of 
$\Delta_{\arith}(\langle P\rangle)$
and $\deg_\cL Y$ and so it suffices to prove the second bound. 

Without loss of generality we may assume $H\not=\grpx$. 
By Poincar\'e's Complete Reducibility Theorem 5.3.5 \cite{CAV} there are 
up-to $\overline K$-isogeny only finitely many possibilities for 
$\grpx/H$. 
So we may assume that there is an abelian variety $\grpy$, coming from
a finite set independent of
$P$, and a  surjective homomorphism $\grpx\rightarrow \grpy$ whose
kernel contains  $H$ as a connected component.
After multiplying  said homomorphism by a positive integer we may assume
$P$ lies in its kernel. 
We observe that 
the assertion of the proposition becomes stronger when enlarging
$K$, so we may assume that $H, \grpx,$ and all elements in $\Hom{\grpx,\grpy}$ are defined over $K$.

We apply Lemma \ref{lem:galoisorbit} to find a surjective homomorphism 
$\varphi:\grpx\rightarrow \grpy$ with $\varphi(P)=0$ and whose norm is bounded from above
by $c_6 D^{6\dim \grpy+1+(2\dim \grpy+9)\rho /2}h^{\rho/2}$
with  
$D=[K(P):K]$ and $\rho>0$ the rank of $\Hom{\grpx,\grpy}$.
We have
$\dim H = \dim X - \dim Y$
by a dimension counting argument. 

Let $\Omega_\grpx\subset \ts{\grpx}$ denote the period lattice and tangent
space of $\grpx$ at the origin. 
We use the same norm $\|\cdot\|$ on $\ts{\grpx}$ as
introduced in Section \ref{sec:degabsv}.
If $\Omega_\grpy\subset \ts{\grpy}$ denotes the period lattice of
$\grpy$, then
$\varphi$
induces a linear map $\Omega_\grpx\rightarrow\Omega_\grpy$. 
Say $g'=\dim H$.

To proceed we apply an adequate version of Siegel's Lemma 
to solve $\varphi(\omega_i) = 0$ 
in
linearly  independent periods $\omega_1,\ldots,\omega_{2g'}\in \Omega_\grpx$ with
 controlled norm. Indeed, 
we may refer to Corollary 2.9.9 \cite{BG}, however the 
numerical constants there will not matter for us.  
Siegel's Lemma yields the first inequality in  
\begin{equation}
\label{eq:prodomegabound}
  \|\omega_1\|\cdots \|\omega_{2g'}\| \le c_7 \|\varphi\|^{2(g-g')}
= c_7\|\varphi\|^{2\dim \grpy}
\le c_8 D^{58g^4}\max\{1,\hat h(P)\}^{\rho \dim \grpy},
\end{equation}
the second one follows from the bound for $\|\varphi\|$ we deduced
further up and
\begin{equation*}
  (12\dim \grpy + 2+(2\dim \grpy + 9)\rho ) \dim \grpy
\le (12g+2+(2g+9)4g^2)g \le 58g^4
\end{equation*}
as $\rho\le 4g\dim \grpy$ and $\dim \grpy\le g$. 

Lemma \ref{lem:deg} and Hadamard's Lemma  yield
$$
[\ker \varphi:H] \deg_\cL H =(\dim
H)![\ker\varphi:H]\vol{\Omega_H}\le
(\dim H)! \|\omega_1\|\cdots \|\omega_{2g'}\|.
$$
As $[\ker\varphi:H]\ge 1$ we get an upper bound for $\deg_\cL H$ which
yields the assertion when combined
with (\ref{eq:prodomegabound}). 
\end{proof}

\subsection{LGO and the N\'eron-Tate height}

We begin by exhibiting a connection between LGO, definition
\ref{def:LGO}, and height upper bounds on
abelian varieties

Let $\grpx$ be an abelian variety defined over
a number field $K$ and suppose $\cL$ is a symmetric, ample line bundle on $\grpx$. 
Let $\hat h$ denote the
associated
 N\'eron-Tate height function.
Observe that any optimal subvariety of a subvariety of $\grpx$
defined over $K$ is defined over $\Kbar$.

 \begin{definition}
Let $\var\subset \grpx$ be a subvariety defined over
$K$. Let $\hgtexp\ge 0$. We define $\opt{\var;\hgtexp}{}$ to be the set of those
$\opta\in \opt{\var}{}$ which contain a $P\in \opta(\overline K)$ with
$\hat h(P)\le (2[K(\opta):K])^\hgtexp$; here $K(\opta)$ denotes the smallest  field
containing $K$ over which $\opta$ is defined. 
 \end{definition}


In order to  apply the counting strategy to study
optimal subvarieties we must find a
polynomial upper bound for the complexity of a torsion coset in terms
of its arithmetic degree. 
Since the inequality in Proposition \ref{prop:boundcomplexity} also
involves the height we make the following observation. 

 \begin{proposition}
\label{prop:hgttolgo}
 Let $\var\subset \grpx$ be a
 subvariety defined over $K$ and let $s\ge 0$. 
Then $\lgo_s(\var)$ is satisfied if  there exist $\epsilon > 0$ and $\hgtexp\ge
0$ such that 
\begin{equation*}
 \hat h(P)\le (2[K(P):K])^{\hgtexp} (\deg_\cL \langle
P\rangle)^{\frac{1}{\lambda_\grpx(\dim\langle P\rangle)} - \epsilon} 
\end{equation*}
 for all optimal singletons
$\{P\}\subset \var$ with $\lambda_\grpx(\dim\langle P\rangle) > 0$ and 
$\dim\langle P\rangle \le s$. 
In particular, $\lgo_s(\var)$ is satisfied if the N\'eron-Tate 
height of an optimal
singleton in $\var$ with defect at most $s$ is bounded from above uniformly. 
 \end{proposition}
 \begin{proof}
Observe that $\lambda_\grpx(\dim\langle P\rangle) = 0$ if and only if
$\langle P\rangle = \grpx$. In this case $\{P\}$ can only be an optimal
singleton of $\var$ if $\defect(\var) = \dim \grpx$. This equality entails
$\var=\{P\}$ in which case the claim is trivial.

If $\lambda_\grpx(\dim \langle P\rangle) > 0$ the claim is direct consequence of Proposition
   \ref{prop:boundcomplexity}
and the definition of $\lgo_s(\var)$. 
 \end{proof}

The main result of this section states that a lower bound for the
Galois orbit as in (\ref{eq:deglb}) is sufficient to prove that there
are only finitely many optimal subvarieties of $\var$.  
Although we believe (\ref{eq:deglb}) to always hold, we are not able
to prove it. However, we can show unconditionally that
$\opt{\var;\hgtexp}{}$ is  finite for any fixed $S\ge 0$.

\begin{theorem}
\label{thm:optfinite}
Let $\grpx$ be an abelian variety defined over a field $K$ which is
finitely generated over $\IQ$. 
Let $\var\subset\grpx$ be a subvariety
defined over $K$. 
\begin{enumerate}
\item [(i)]
Say $r,s\ge 0$ and suppose that all quotients of $X$ defined over a
finite extension of $K$ satisfy 
$\lgo_s^r$. 
Then
\begin{equation}
\label{eq:optrs}
  \left\{ \opta\in \opt{\var}{} ;\,\, \codim_\var \opta \le r \text{ and }\dim
  \langle \opta\rangle-\dim\geo{\opta} \le s  \right\}
\end{equation}
is finite. 
\item[(ii)]
If $K$ is a number field, then $\opt{\var;\hgtexp}{}$ is finite
for all $\hgtexp\ge 0$.
\end{enumerate}
\end{theorem}

We obtain $2$ immediate corollaries.

\begin{corollary}
  \label{cor:unlikelyabvarQbar}
Let us suppose that the height bound in Proposition \ref{prop:hgttolgo} holds for
all subvarieties of all abelian varieties   defined over any
number field. Then the Zilber-Pink Conjecture \ref{conj:ZPatypical}
holds for all subvarieties of all abelian varieties defined over any  number field. 
\end{corollary}
\begin{corollary}
Let $\grpx$ and $\var$ be as in Theorem \ref{thm:optfinite}. We suppose that
 all quotients of 
$\grpx$ defined over a finite extension of $K$ satisfy $\lgo_s^r$ for all $r,s\ge 0$. Then $\opt{\var}{}$ is
finite for any  subvariety $\var$ of $\grpx$ defined
over $\overline K$. 
\end{corollary}

Let us look more closely at the case when $K$ is a number field and
$s$ is small. 
The N\'eron-Tate height of a torsion point vanishes. 
So by Proposition \ref{prop:hgttolgo}  any abelian variety over a number field satisfies
$\lgo^r_0$ for all $r\ge 0$. Part (i) of the theorem implies that $\var$ contains
only finitely many maximal torsion cosets  as such subvarieties
are necessarily optimal. We recover the conclusion of the
Manin-Mumford Conjecture. Of course in this special case,  our argument does not differ
significantly from the Pila-Zannier approach \cite{PilaZannier}.
But part (i) of our theorem applied to $s=0$ even yields the following
corollary.

\begin{corollary}
Let $\grpx$ and $\var$ be as in Theorem \ref{thm:optfinite}(i) with
$K$  a number field.
Then
\begin{equation*}
  \{\opta\in\opt{X}{};\,\, \langle\opta\rangle = \geo{\opta} \}
\end{equation*}
is finite. 
\end{corollary}


The next  case is   $s=1$; the corresponding case of the Zilber-Pink
Conjecture concerns subvarieties of codimension $2$. So say
$\{P\}\subset \var$ is an optimal singleton with
$\dim\langle P\rangle \le 1$. 
If $\dim \langle P\rangle=0$, then $P$ is of finite order and we are
back in the case $s=0$. 
So we assume $\dim \langle P\rangle = 1$.
We know that $\{P\}$ is geodesic-optimal with respect to $\var$ by Proposition
\ref{prop:geodesicopt}. In other words, $P$ is not contained in a
coset of positive dimension contained completely in $\var$. In this setting
it would be interesting to know if  $\hat h(P)$ is bounded from above in terms of $\var$
only. 
The analogous statement in the context of algebraic tori was proved by Bombieri and Zannier, cf. Theorem 1 \cite{ZannierAppendix}. 
Moreover,  Checcoli, Veneziano,
and Viada \cite{CheVenVia13} showed a related result 
 inside a product of elliptic curves with complex multiplication.

\begin{proof}[Proof of Theorem \ref{thm:optfinite}]
We  can almost prove parts (i) and (ii)  simultaneously. However, at
times we will branch off the
main argument to specialise to the two statements. 
  The proof will be by induction on $\dim \var$. Our theorem is trivial
  if $\var$ is a point. Say $\dim \var \ge 1$. 
After replacing $K$ by a finite extension we may suppose that $\var$,
 all
abelian subvarieties of $\grpx$, and all relevant homomorphisms below are defined over $K$. 
We may assume that  $\Kbar$ is a subfield of $\IC$.
We fix $\cL$  an ample line bundle on $\grpx$ 
to make sense of the complexity $\Delta(\cdot)$. 

Suppose $\opta$ is an element of $ \opt{\var}{}$ or $\opt{\var;\hgtexp}{}$
depending on whether we are in case (i) or (ii) of the theorem. 

In case (i) we suppose  $\codim_\var\opta \le r$ and 
$\dim\langle \opta\rangle -\dim\geo{\opta}\le s$;
in case (ii) we suppose that $\opta$ contains a point of height at
most $(2[K(\opta):K])^{\hgtexp}$. 

By Proposition \ref{prop:geodesicopt} the subvariety $\opta$ is
geodesic-optimal and thus an irreducible
component of $\var\cap\geo{\opta}$. By Proposition \ref{prop:geofinite} the coset $\geo{\opta}$ is the translate of an abelian subvariety $\grpy\subset \grpx$
that comes from a finite set depending only on $\var$. 
We observe that this finiteness statement is trivial if  $\dim\var=1$;
 we do not require Proposition \ref{prop:geofinite} if $\var$ is a
 curve. 
We will also fix an ample and symmetric line bundle on the abelian
variety
 $\grpx/\grpy$ in order to
speak of  the N\'eron-Tate height $\hat h$.

Let 
 $\varphi:\grpx\rightarrow \grpx/\grpy$ be the canonical morphism.
As we are in characteristic $0$, there is a Zariski open and dense
subset $\var'\subset \var$ such that
$\varphi|_{\var'}:\var'\rightarrow\varphi(\var')$ is a smooth morphism of 
relative dimension $n$
and $\varphi(\var')$ is Zariski open in $\overline{\varphi(\var')}$, 
 cf. Corollary III.10.7 \cite{Hartshorne}. 

If $\opta\cap \var'=\emptyset$, then $\opta$ is contained in an irreducible component
 of $\var\ssm \var'$ and $\opta$ is an optimal subvariety of this irreducible
 component. 
In both  case (i) and (ii) we may apply induction on the dimension as
 $\dim (\var \ssm \var') < \dim \var$; for case (i) we observe that
the codimension in (\ref{eq:optrs}) drops.
So there are only finitely many possibilities for $\opta$.

Let us assume  $\opta\cap \var'\not=\emptyset$. 
We note that $\opta\cap \var'$ is an irreducible component of a
 fibre of $\varphi|_{\var'}$. 
The fibres of $\varphi|_{\var'}$ are equidimensional of dimension $n$. 
So 
$\dim \opta=n$  and
 $\varphi$ maps $\opta$ to some  $P\in (\grpx/\grpy)(\IC)$.

Since $P$ lies in the torsion coset $\varphi(\langle
\opta\rangle)$ we find $\langle P\rangle \subset\varphi(\langle \opta\rangle)$ and thus $\varphi^{-1}(\langle P\rangle)
\subset \varphi^{-1}(\varphi(\langle \opta \rangle))\subset \langle
\opta\rangle + \grpy$.
But $\grpy$ is contained in a translate of $\langle \opta\rangle$ and
thus
$\langle\opta\rangle+\grpy = \langle\opta\rangle$. We conclude
\begin{equation}
\label{eq:dimineq2}
\dim Y+\dim\langle P\rangle  =  \dim  \varphi^{-1}(\langle P\rangle) \le \dim \langle \opta\rangle.
\end{equation}

Next we claim that the singleton $\{P\}$ is an optimal subvariety of
${\varphi(\var)}$. If the contrary holds there is a subvariety
$\optb'$ of ${\varphi(\var)}$  
containing $P$, with positive dimension, and defect
at most $ \dim \langle P\rangle$. 
We fix an irreducible  component $\optb$, that meets $V'$, of the pre-image of $\optb'$ under
$\varphi|_{\var}$ with 
$\opta\subset \optb$. As $\varphi|_{\var'}$ is smooth of relative dimension
$n=\dim \opta$ we have 
\begin{equation}
\label{eq:dimDDprime}
\dim \optb = \dim \optb'+\dim \opta > \dim \opta.  
\end{equation}
 We remark 
$\langle \optb\rangle \subset \varphi^{-1}(\langle \optb'\rangle)$,
so $\dim \langle \optb\rangle \le \dim \grpy +\dim \langle \optb'\rangle$. 
Since $\defect(\optb')\le \dim \langle P\rangle$ we find
\begin{equation*}
  \dim \langle \optb\rangle \le \dim \grpy + \dim \optb'+\dim \langle P\rangle. 
\end{equation*}
 Optimality of  $\opta$  and $\optb\supsetneq \opta$ from (\ref{eq:dimDDprime})  imply 
\begin{alignat*}1
 \dim \langle \opta\rangle &<  \dim \opta + \dim \langle \optb\rangle -\dim \optb
\le \dim \opta +\dim \grpy +\dim \optb' - \dim \optb +\dim \langle P\rangle.
\end{alignat*}
We use the equality in (\ref{eq:dimDDprime}) to find
$\dim \langle \opta\rangle < \dim \grpy +\dim \langle P\rangle$.
This contradicts
 (\ref{eq:dimineq2}) and so $\{P\}$ must be an optimal subvariety of 
${\varphi(\var)}$.

Let us suppose  $\dim \opta > 0$ for the moment.  Then 
$\dim {\varphi(\var)} = \dim \var - \dim \opta  <
\dim \var$.

We first branch into case (i). 
Any singleton in $\varphi(\var)$ 
has codimension  $\dim \varphi(\var) =\codim_\var \opta\le r$.
The bound (\ref{eq:dimineq2}) and $\dim \grpy = \dim \geo{\opta}$ together yield
$\dim\langle P \rangle\le \dim\langle \opta\rangle - \dim\geo{\opta}$. So
$\dim\langle P\rangle - \dim\geo{P} = \dim\langle P\rangle \le s$ by (\ref{eq:optrs}).  
As $\dim\varphi(\var) < \dim \var$  
 there are only finitely many
possibilities for $P$ by induction. Recall that
 $\varphi|_{\var'}^{-1}(P)$
contains $A\cap\var'$ as an irreducible component.
This leaves at most finitely many possibilities for $\opta$.

In case (ii) we also want to use induction, but doing so requires a
control of the height. By definition there exists   $P'\in \opta(\overline K)$ with
$\hat h(P') \le (2[K(\opta):K])^{\hgtexp}$. 
Recall that  $\varphi$ comes from a finite set depending only on $\var$.
Now $P=\varphi(P')$ and by properties of
the N\'eron-Tate height we have 
\begin{equation}
\label{eq:hgtauxpt}
  \hat h(P) = \hat h(\varphi(P')) \le c_1 \hat h(P')
\le c_1(2[K(\opta):K])^{\hgtexp}
\end{equation}
here and below $c_1,c_2,\ldots$ are positive constants that depend
only on $\grpx,\var,$ and $\cL$ but not on $\opta, P,$ or $P'$.
By Bertrand's Theorem, which we already used in the proof of Lemma
\ref{lem:arithlb},
there exists an 
abelian subvariety $\grpz\subset \grpx$ such that
 $\varphi|_\grpz :\grpz \rightarrow \grpx/\grpy$ is an isogeny of degree at most
$c_2$. As $\varphi$ is defined over $K$ we have
$[K(P''):K] \le c_2[K(P):K]$ for any $P''\in \grpz$ with $\varphi(P'')=P$. 
The intersection $\var\cap(P'+\grpy)=\var\cap (P''+\grpy)$
contains $\opta$ as an irreducible component. 
Say $\sigma\in {\rm Gal}(\overline K/K)$, then
$\sigma(\opta)$ is an irreducible component
of $\var\cap (\sigma(P'')+\grpy)$.
As the number of components of this intersection is at most
 a constant depending only on $\var$ and $\grpy$ we have
$[K(\opta):K] \le c_6[K(P''):K]\le c_2c_6[K(P):K]$. Inequality
(\ref{eq:hgtauxpt}) yields
\begin{equation*}
  \hat h(P) \le c_1(2c_2c_6 [K(P):K])^{\hgtexp}. 
\end{equation*}
So $\{P\}\in\opt{\varphi(\var);\hgtexp'}{}$ 
 for an appropriately chosen
$\hgtexp'$. 
As in (i) we conclude by 
induction that $\opta$ is in a finite set depending only
on $\var$. 

It remains to treat the case $\dim \opta=0$. 
Then $\grpy=0$, 
 $\varphi$ is the identity,
and $\opta$ consists only of $P$. 
Thus $\{P\}\subset\var$ is an optimal singleton. 

In case (i) we first note $\dim \var = \codim_\var \{P\} \le r$. So
$\lgo_s(\var)$ holds
as $\grpx$ satisfies $\lgo_s^r$. We get
\begin{equation}
\label{eq:Deltabound}
 \Delta(\langle P\rangle) \le (2[K(P):K])^\kappa
\end{equation}
where $\kappa >0$ 
 depends only on $\grpx$ because
$\dim\langle P\rangle = \dim\langle P\rangle -\dim\geo{P}\le s$.

In case (ii) we note that
$\{P\}\in\opt{\var;\hgtexp}{}$ implies the height bound
$\hat h(P) \le (2[K(P):K])^{\hgtexp}$. 
We use this bound in connection with 
Proposition \ref{prop:boundcomplexity}
 and arrive again at an inequality of the
form 
(\ref{eq:Deltabound}). 

So in any case, we have found a lower bound for the size of the Galois
orbit of $P$. 
Next we set the stage for the o-minimal machinery. All choices
 in the following paragraphs  are made independently of $P$ unless
 stated otherwise. 

Let us fix a basis $\omega_1,\ldots,\omega_{2g}$ of the period
lattice
$\Omega_\grpx\subset \ts{\grpx}$ as in Section
\ref{sec:finitegeoabvar}, here $g=\dim \grpx$.
We use it to 
 identify $\ts{\grpx}$ with $\IR^{2g}$ as an $\IR$-vector space.
In these coordinates, 
 $\exp :\IR^{2g}\rightarrow \grpx(\IC)$ is a real-analytic group
homomorphism  with kernel $\IZ^{2g}$. 

By the discussion above the set
\begin{alignat*}1 
  \defF = \{(\psi,w,z)\in \mat{2g}{\IR}\times\IR^{2g}\times\IR^{2g};\,\,
z\in \exp|_{[0,1)^{2g}}^{-1}(\var(\IC)) \text{ and } 
\psi(z-w)=0
\} 
\end{alignat*}
is definable in $\IRan$
where we identify $\mat{2g}{\IR}$ with $\IR^{(2g)^2}$. 
 We will consider $\defF$ as a definable family parametrised by
 $\mat{2g}{\IR}$. The kernel of each
matrix in $\mat{2g}{\IR}$ is 
a vector subspace of $\IR^{2g}$.
In our application, the kernel will be the tangent space of the abelian
subvariety determined by  $\langle P\rangle$. 

Indeed, let us write $\langle P\rangle = Q+\grpz$ 
with $\grpz$ an abelian subvariety of $\grpx$ and 
where $Q$ has minimal finite order $N$, i.e. 
 $\Delta(\langle P\rangle) = \max\{N,\deg_\cL \grpz\}$. 
As opposed to $\grpy$, we
 do not yet know that $\grpz$ comes from a finite set; so we must keep
 in mind that $Q$ and $\grpz$
 depend on $P$.

Let $\sigma \in {\rm Gal}(\overline K/K)$, then
$\sigma(P) \in \sigma(Q) + \grpz$. We
write $\sigma(P) = \exp(z_\sigma)$ 
and $\sigma(Q) = \exp(q_\sigma)$ 
with
$z_\sigma,q_\sigma\in [0,1)^{2g}$.
Now $\exp(z_\sigma-q_\sigma)=\sigma(P-Q) \in \grpz$ implies
$z_\sigma-q_\sigma \in \exp^{-1}(\grpz) = \Omega_\grpx+\ts{\grpz}$. 

Let $\|\cdot\|$ be a norm on $\ts{\grpx}$ as in Section \ref{sec:degabsv}. 
According to Lemma \ref{lem:geonumbs}(ii) there exists
$\omega_\sigma \in\IZ^{2g}$ with $z_\sigma-q_\sigma - \omega_\sigma
\in \ts{\grpz}$ and
\begin{equation*}
  \| \omega_\sigma \| \le \|z_\sigma-q_\sigma\| + c_4 \deg_\cL \grpz.
\end{equation*}
But $\|z_\sigma\|\le c_5$ and $\|q_\sigma\|\le c_5$
as these elements are in the bounded set $[0,1)^{2g}$. Thus
$\|\omega_\sigma\|\le c_6 \deg_\cL \grpz$. Now
$\omega_\sigma$ is integral, hence
$\Height{\omega_\sigma}\le c_7 \deg_\cL \grpz$
where the height is as in Section \ref{sec:semirational} and $c_7\ge 1$.

As $\sigma(Q)$ has order $N$ we find
$q_\sigma \in \frac 1N \IZ^{2g}$.
The coordinates of $q_\sigma$ lie in $[0,1)$ and so
$\Height{q_\sigma}\le N$. 

Basic height properties yield
$\Height{q_\sigma+\omega_\sigma}\le
2\Height{q_\sigma}\Height{\omega_\sigma}$. So
\begin{equation*}
  \Height{q_\sigma+\omega_\sigma}
\le 2c_7 N \deg_\cL \grpz \le 2c_7
\Delta(\langle P\rangle)^2. 
\end{equation*}

The tangent space $\ts{\grpz}\subset \ts{\grpx}$ is the kernel
of some $\psi\in \mat{2g}{\IR}$. 
By construction, $(q_\sigma+\omega_\sigma,z_\sigma)$
lies on the fibre $\defF_\psi$. The number of distinct 
 $z_\sigma$ is $[K(P):K]$ which is bounded from below by (\ref{eq:Deltabound}).
 The $q_\sigma+\omega_\sigma$ are rational
with height at most $T = 2c_7
\Delta(\langle P\rangle)^2 \ge 1$.

There are only finitely many torsion cosets $\optb$ for with
$\Delta(\optb)$ is bounded by a constant.  The singleton $\{P\}$, being optimal,  is an irreducible component of
$\var\cap\langle P\rangle$. 
So we will assume that
 $\Delta(\langle P\rangle)$ is  sufficiently large
with respect to the fixed data.
Under this hypothesis and with for example $\epsilon = 1/(4\kappa)$
we can apply Corollary \ref{cor:relation}. We proceed to show that
this leads to  a
contradiction.

There is
 $\beta:[0,1]\rightarrow \defF_\psi$ as in Corollary \ref{cor:relation}
with $\Sigma$ the set $\{(q_\sigma+\omega_\sigma,z_\sigma);\sigma \in
{\rm Gal}(\overline K/K)\}$. 
The o-minimal structure $\IRan$ admits analytic cell decomposition by
a result of van den Dries and Miller \cite{DriesMiller:94}. 
So we may assume that $\beta$ is real analytic on
 $(0,1)$.
The first projection $\beta_1=\pi_1\circ\beta:[0,1]\rightarrow\ts{\grpx}$ is
semi-algebraic and the second one
$\beta_2=\pi_2\circ\beta:[0,1]\rightarrow\ts{\grpx}$ is non-constant.
The path $\beta_2$ begins at $\beta_2(0) = z_\sigma$
for some $\sigma\in{\rm Gal}(\overline K/K)$.
The image of $\beta_2-\beta_1$ lies  in $\ts{\grpz}$ by our choice of
$\psi$.
So $\phi\circ\exp\circ\beta_1 = \phi\circ\exp\circ\beta_2$ where 
 $\phi:\grpx\rightarrow \grpx/\grpz$ is the quotient morphism.

We claim that $\phi\circ\exp\circ\beta_1$ is non-constant. Let us
assume the contrary, then $\phi\circ\exp\circ\beta_2$ is constant too.
As $\exp\circ\beta_2(0) = \sigma(P)$
we have
 $(\exp\circ\beta_2)([0,1]) \subset \sigma(P) + \grpz = 
\sigma(\langle P\rangle)$.
But the said image lies in $\var(\IC)$ by the definition of $\defF$.
So $(\exp\circ\beta_2)([0,1]) \subset \var\cap\sigma(\langle P\rangle)=
\sigma(\var\cap \langle P\rangle)$.
By the arguments above,
Recall that $\{P\}$ is an optimal singleton of $V$.
So $P$ is isolated in  $\var\cap\langle P\rangle$
and thus  $\sigma(P)$ is isolated in 
$\sigma(\var\cap \langle P\rangle)$. 
This contradicts the fact that $\beta_2$ is continuous and
non-constant.

Let $R\subset \grpx(\IC)$ denote the image of $\exp\circ\beta_1$,
it is an uncountable set 
by the previous paragraph.
The differential $d\phi:\ts{\grpx}\rightarrow\ts{\grpx/\grpz}$ of $\phi$ is a linear map. 
So $\phi(R)$ is the image of the semi-algebraic
map $d\phi\circ \beta_1$ composed with $\ts{X/Z}\rightarrow (X/Z)(\IC)$.
The Ax-Lindemann-Weierstrass Theorem \ref{thm:ax2} implies that 
$\phi(\zcl{R})=\zcl{\phi(R)}\subset \grpx/\grpz$ 
 is a positive dimensional coset. We abbreviate $\optc  =\sigma^{-1}(\zcl{R})$
which contains $P$ as a point. 
Then $\optc$ must be irreducible, as $\beta_1|_{(0,1)}$ is real
analytic,  and of positive dimension. 
The image 
$\phi(\optc)$ is a translate of $\grpz'/\grpz$ where
$\grpz'\subset \grpx$ is an abelian subvariety that contains $\grpz$.  
Now $\optc$ is contained in the  coset $\phi^{-1}(\phi(\optc)) = P +
\grpz'$. This 
 is even a torsion coset because 
$P+\grpz' \supset P+\grpz = \langle P\rangle$ contains a point of
 finite order. Therefore,
 \begin{equation*}
   \langle \optc\rangle \subset P+\grpz'.
 \end{equation*}

Basic dimension theory yields the inequality in 
$\dim \grpz'/\grpz =  {\phi(\optc)}
\le \dim \optc$.
Thus
\begin{equation*}
\defect(\optc) =\dim \langle \optc\rangle - \dim \optc \le\dim(P+\grpz')-\dim \optc \le
\dim \grpz' - \dim \grpz'/\grpz = \dim \grpz  = \defect(P).  
\end{equation*}
But $\dim \optc \ge 1$ and
 this contradicts the optimality of 
$\{P\}$.
\end{proof}

\subsection{Intersecting with algebraic subgroups}
\label{sec:resultsabvar}

In this section we prove that  a height upper bound
for curves due to  R\'emond  in combination with 
the o-minimal machinary 
is strong enough to establish 
$\lgo^1_s$ for all abelian varieties and all $s\ge 0$. In turn this
will yield 
Theorem \ref{thm:curvesabvar}, 
the Zilber-Pink
Conjecture for curves in abelian varieties when all geometric objects
are defined over an algebraic closure of the rationals. 
We also prove some partial results in the direction of this conjecture
for higher dimensional subvarieties.

\begin{theorem}[R\'emond]
\label{thm:remond}
Let $\grpx$ be an abelian variety defined over a number field $K$, equipped
with an ample, symmetric line bundle and its associated N\'eron-Tate
height.
Suppose that
$\var$ is a  curve in $\grpx$ that is not contained in any
proper abelian subvariety of $\grpx$. 
Then the N\'eron-Tate height is bounded from above on $\var(\overline K)\cap \sgu{\grpx}{2}$. 
\end{theorem}
\begin{proof}
  This is R\'emond's corollaire 1.6 \cite{RemondInterII}.
\end{proof}

\begin{corollary}
  Any abelian variety defined over a number field satisfies $\lgo^1_s$
  for all $s\ge 0$. 
\end{corollary}
\begin{proof}
Let $\var\subset \grpx$ be a subvariety with $\dim \var\le
1$. 
We may clearly assume that $\var$ is a curve. 
  If $\{P\}\subset \var(\Kbar)$ is an optimal singleton, then 
$\dim \langle P\rangle  = \defect(\{P\}) < \defect(\var)= \dim \langle \var\rangle - 1$. 
After  translating $\var$ by a torsion point it is contained in the
abelian subvariety of
$\grpx$  determined by $\langle \var\rangle$.
Without loss of generality 
it suffices to verify $\lgo^1_s$ with  $\grpx$ replaced by this abelian subvariety.  
Now $\langle \var\rangle = \grpx$ and 
 $P$ is contained in an abelian subvariety of $\grpx$ of 
codimension at least
$2$. 
So the height of $P$ is bounded from above 
by R\'emond's Theorem.
On inserting this height bound in 
Proposition \ref{prop:boundcomplexity}  we find
that $\lgo_s(\var)$ is satisfied. 
\end{proof}


We combine the  corollary with
results from the previous section to obtain  
 the following
strengthening
of Theorem \ref{thm:curvesabvar}. 
\begin{theorem}
\label{thm:abvar1}
Let $\grpx$ be an abelian variety defined over a number field $K$. Suppose 
  $\var\subset \grpx$ is a subvariety defined over $K$. 
  \begin{enumerate}
  \item [(i)] The set 
    \begin{equation*}
      \left\{\opta\in\opt{\var}{};\,\, \codim_\var \opta \le 1 \right\}
    \end{equation*}
is finite. 
\item[(ii)] If $\var$ is a curve then $\opt{\var}{}$ is finite. 
  \item[(iii)] If $\var$ is a curve that is not contained in a proper abelian subvariety
    of $\grpx$, then  $\var(\Kbar)\cap \sgu{\grpx}{2}$ is finite. 
  \end{enumerate}
\end{theorem}
\begin{proof}
Part (i) follows from the corollary above and from Theorem \ref{thm:optfinite}(i). 
  Part (ii) is a special case of (i) and to see (iii) we 
 note that any point in the  intersection
defines an optimal singleton of $\var$. 
\end{proof}

The next theorem improves on Theorem
\ref{thm:higherdimabvar} stated in the introduction.
We require Definition \ref{def:anomalous} in part (iii)  below. 
Parts (iii) and (iv) rely on a height bound of the first-named author whereas
parts (v)  and (vi) use a result of R\'emond.

\begin{theorem}
\label{thm:abvar2}
Let $\grpx$ be an abelian variety defined over a number field $K$, equipped
with an ample, symmetric line bundle and its associated N\'eron-Tate
height $\hat h$.
Suppose  $\var\subset \grpx$ is a subvariety defined over $K$. 
\begin{enumerate}
\item [(i)]
If $\hgtexp\ge 0$ then
\begin{equation*}
  \left\{P \in \var(\Kbar)\cap \sgu{\grpx}{1+\codim_\grpx \langle \var\rangle + \dim \var};\,\, \hat h(P)\le \hgtexp\right\}
\end{equation*}
is not Zariski dense in $\var$.
\item[(ii)]
If $\hgtexp\ge 0$ then
\begin{equation*}
  \left\{P \in \var(\Kbar)\cap \sgu{\grpx}{1+ \dim \var};\,\, \hat h(P)\le \hgtexp\right\}
\end{equation*}
is contained in a finite union of proper algebraic subgroups of $\grpx$. 
\item[(iii)] The set $\oa{\var}(\Kbar)\cap \sgu{\grpx}{1+\dim \var}$ is
  finite. 
\item[(iv)]
Suppose $\dim \var \ge 1$ and 
 $\dim \varphi(\var) = \min \{\dim \grpx/\grpy, \dim
  \var\}$ for all abelian subvarieties $\grpy\subset \grpx$ where
  $\varphi:\grpx\rightarrow \grpx/\grpy$ is the canonical morphism. Then 
  $\var(\Kbar)\cap \sgu{\grpx}{1+\dim \var}$ is not Zariski dense in
  $\var$. 
\item[(v)] 
 Let $\Gamma \subset
  \grpx(\overline K)$ be a finitely generated subgroup and
  \begin{equation*}
    \overline\Gamma = \{P\in \grpx(\overline K);\,\, \text{there is an integer
      $n\ge 1$
with } nP\in\Gamma\}.
  \end{equation*}
Then
    \begin{equation}
\label{eq:gammaalgsbgrp}
        \bigcup_{\substack{H\subset \grpx \\ \text{$H$
            algebraic subgroup} \\ \codim_\grpx
          H\ge 1+\dim \var}} \oa{\var}(\Kbar)\cap (H+\overline\Gamma).
    \end{equation}
is finite.
\item[(vi)] Let $\var$ be as in (iv) and $\overline\Gamma$ be as in
  (v).
Then
\begin{equation*}
 \bigcup_{\substack{H\subset \grpx \\ \text{$H$
            algebraic subgroup} \\ \codim_\grpx
          H\ge 1+\dim \var}}       \var(\Kbar)\cap (H+\overline\Gamma)
    \end{equation*}
is not Zariski dense in $\var$.
\end{enumerate}
\end{theorem}
\begin{proof}
  For part (i) let $P\in \var(\Kbar)\cap\sgu{\grpx}{1+\codim_X\langle
    V\rangle + \dim \var}$ with $\hat
  h(P) \le \hgtexp$. We can
  enlargen
$\{P\}$ to an optimal subvariety $\opta$ of $\var$ with 
$\defect(\opta)\le \dim \langle P\rangle$. 
As $\opta$ contains a point of height at most  $\hgtexp\le (2[K(\opta):K])^{\hgtexp}$ we
find $\opta\in \opt{\var;\hgtexp}{}$. 
But the latter set is finite according to Theorem
\ref{thm:optfinite}(ii). To prove (i) it suffices to establish
$\opta\not=\var$. This follows from  
$\defect(\opta) \le \dim \langle P\rangle  < \dim \langle \var\rangle - \dim
\var$.

Part (ii) is a consequence of (i). Indeed, there is nothing to show if
$\var$ is contained in a proper algebraic subgroup of $\grpx$ or if
$\var$ is a point. Otherwise we
have
$\langle \var\rangle = \grpx$ and by (i) the set in the assertion is contained in 
$\var_1\cup\cdots\cup \var_n$ where $\var_i\subsetneq \var$ are 
subvarieties of $\var$. We conclude (ii) by  induction  on $\dim \var > \dim \var_i$.

To show (iii) we may assume $\var\not=\grpx$ and 
that $\var$ is not contained in a proper
abelian subvariety of $\grpx$. 
Indeed, if the second condition fails, then $\oa{\var}=\emptyset$.
We know from the main result of  \cite{abvar}
that the N\'eron-Tate height is bounded from above by $\hgtexp$, say,  on
$\oa{\var}(\Kbar)\cap\sgu{\grpx}{\dim \var}$ and thus in particular on 
$\oa{\var}(\Kbar)\cap \sgu{\grpx}{1+\dim \var}$. 
R\'emond \cite{RemondInterII,RemondInterIII} independently obtained a height on
this set, in his notation we have 
$\var\ssm \oa{\var}=Z_{\var,{\rm an}}^{(1+\dim\var)}$.
Say $P$ is in this intersection.
If $\opta$ is an optimal subvariety of $\var$
containing $P$ with $\delta(\opta)\le \dim\langle P\rangle$
then 
$\dim \grpx - \dim \var - 1 \ge \dim\langle P\rangle \ge
 \dim\langle \opta\rangle - \dim \opta$.
But $P\in \oa{\var}(\Kbar)$  entails $\dim \opta = 0$. So
$A=\{P\}\in\opt{\var;\hgtexp}{}$ and the (iii) follows as (i) because 
$\opt{\var;\hgtexp}{}$ is finite. 

We recall that $\oa{\var}$ is Zariski open in $\var$. 
Part (iv) follows from (iii) since the condition on $\var$ entails
$\oa{\var}\not=\emptyset$
by the final comment in Section \ref{sec:finitegeoabvar}.  

For part (v) we will require R\'emond's height bound in his Th\'eor\`eme 1.2
\cite{RemondInterII}
combined with his Th\'eor\`eme 1.4 \cite{RemondInterIII}. 
Together they imply that
there is an upper bound for the N\'eron-Tate  height of the points in
$\oa{\var}(\Kbar)\cap(\overline\Gamma+\sgu{\grpx}{1+\dim \var})$.

We proceed by following the argumentation in the proof of Pink's
Theorem 5.3 \cite{Pink} which we briefly sketch. Suppose
$P_1,\ldots,P_t$ are $\IZ$-independent elements that generate a subgroup of
finite index in $\Gamma$. After replacing $(P_1,\ldots,P_t)\in
\grpx^t(\Kbar)$ by a positive multiple, we may assume that this point
generates a subgroup of $\grpx^t(\Kbar)$ whose Zariski closure in $\grpx^t$ is
an abelian subvariety $\grpy$. 
Now any point $P$ in  (\ref{eq:gammaalgsbgrp})
is a $\IZ$-linear
combination of the $P_i$
 modulo an algebraic subgroup of codimension at least $\dim V'+1$. 
So the augmented point $P'=(P,P_1,\ldots,P_t)$ lies in $V'(\Kbar)\cap
\sgu{(\grpx\times\grpy)}{\dim V'+1}$
because $\dim V'=\dim V$. 
The N\'eron-Tate height of $P'$ is bounded by a constant $S$ that only
depends on $\var$ and the $P_i$.

We proceed as in the end of (iii). Let $\opta'$ 
be an optimal subvariety of $\var'$ that contains the point $P'$
and with $\delta(\opta')\le \delta(\{P'\})$.
So $\dim\langle\opta'\rangle - \dim\opta' \le
\dim\langle P'\rangle \le \dim\grpx\times\grpy - \dim V' -1$.
The projection of $\langle\opta'\rangle\subset\grpx\times\grpy$ to $\grpy$ is an irreducible
component of an algebraic group which contains
$(P_1,\ldots,P_t)$; so it must equal $Y$.
Therefore, each fibre of this projection 
is a coset of dimension
$ \dim \langle\opta'\rangle-\dim Y$.
We observe that $\opta' = \opta\times \{(P_1,\ldots,P_t)\}$
is contained in such a fibre, thus
\begin{equation*}
 \dim \geo{\opta}\le \dim\langle\opta'\rangle-\dim \grpy \le
\dim\opta' + \dim\grpx -\dim\var' -1
=\dim\opta + \dim\grpx -\dim\var -1
\end{equation*}
and hence
\begin{equation*}
\dim\opta\ge \dim\var + \dim\geo{\opta}-\dim\grpx + 1.  
\end{equation*}
Recall that $P\in\opta(\Kbar)$; we must have $\dim\opta=0$ because
$P\in\oa{\var}(\Kbar)$. Thus $A'=\{P'\}$.
Part (v) follows as $A'$ lies in the set $\opt{\var';S}{}$ which is finite by
Theorem \ref{thm:optfinite}(ii).

The proof of (vi) is similar to the proof of (iv). 
\end{proof}

\section{Unlikely intersections in $Y(1)^n$}
\label{sec:unlikelyY1n}


%
%
\begin{theorem}
\label{thm:LGOWCAimplyZPrefined}
 Assume LGO and WCA for $Y(1)^n$.
Let $X\subset Y(1)^n$ be a special subvariety and $V\subset X$.
Then ${\rm Opt}(V)$ is a finite set.
\end{theorem}
\begin{proof}
We prove the theorem by induction on $V$, the case 
$\dim V=0$ being trivial. So we assume that $\dim V\ge 1$ and the
theorem holds 
for all $V$ of smaller dimension.
Let $K$ be a field of definition for $V$ which is finitely generated 
over $\QQ$.


By Proposition 4.5, an optimal component is geodesic-optimal.
By Proposition \ref{prop:geofinitenessY1n} 
the set of ``basic special subvarieties'' that have a translate 
which is  geodesic-optimal is finite.
So the subvarieties comprising ${\rm Opt}(V)$
are components of the intersection of $V$ with the translates 
of finitely many basic special subvarieties $T\subset X$. 
One such $T$ may of course
be the whole of $Y(1)^n$, with $X_T$ being also $Y(1)^n$ 
parameterising individual points of $Y(1)^n$.

Fix such $T$. It evidently suffices now to show that only finitely 
many translates of $T$ are such that $V\cap T$ has components 
which are optimal.  Let $X_T$ denote the translate space 
of $T$, which is a suitable $Y(1)^m$.

We have a ``quotient map'' $\phi: T\rightarrow Y(1)^m$, the fibre over 
$t\in Y(1)^m$ being $T_t$. 
Write $\tau$ for the dimension of these fibres.
By [\HARTSHORNE, Corollary III.10.7], as we are in characteristic 0,
there is a Zariski-open (in $V$) and dense $V'\subset V$ on which 
the restriction  $\phi\vert_{V'}: V'\rightarrow \phi(V')$ is a smooth 
morphism of relative dimension $\nu$. We may further assume that
$\phi(V')$ is Zariski open in its Zariski closure, which we denote 
$V_T\subset X_T$.

Now suppose $C$ is an optimal component of dimension $d$ and 
geodesic defect $\delta$, a component of $V\cap T_y$ for some $y$.
If $C\cap V'=\emptyset$ then $C$ is contained in an irreducible
component of $V-V'$, and is an optimal component of this irreducible
component. As this irreducible component has lower dimension than 
$V$ we conclude by the induction on $\dim V$ 
that there are only finitely many such $C$.

So we may suppose $C\cap V'\ne\emptyset$. Since $C\cap V'$ is an 
irreducible component of a fibre of $\phi\vert_{V'}$ we have 
$n=\dim C$;
also the image of $C$ in $V_T$ under $\phi$ is the point $y$. 
We further
observe that if $y\in A\subset V_T$ is contained in a special subvariety 
$S\subset X_T$
then $\phi^{-1}(A)$ is contained in a special subvariety $\phi^{-1}(S)$
of dimension $\dim S+\tau$. If we take $A$ to be a component
of $\phi\vert_{V'}^{-1}$ containing $C$ we see that
$$
\delta(A)\le \dim S+\tau-\dim A=\dim S+\tau-(\dim A'+\dim C).
$$

Next we claim that $\{y\}$ is an optimal subvariety of $V_T$. 
Note that $\phi(\langle C\rangle)$ is special of dimension
$\dim C+\delta(C)-\tau$ and contains $y$.
Now suppose that $A$ is a component of $V_T$ with 
$\{y\}\subset A$ and
$$
\dim \langle A\rangle -\dim A\le \dim C+\delta(C)-\tau.
$$
Let $B$ be the component of $\phi^{-1}(A)$ containing $C$. Then
$$
\dim B=\dim \big(\phi^{-1}(A)\cap V'\big)=\dim A+\nu
$$
and
$$
\delta(B)\le \dim \phi^{-1}\langle A\rangle -\dim B
\le \dim A+\dim C+\delta(C)-(\dim A+\nu) = \delta(C)
$$
and so by optimality of $C$ we must have $B=C$ and $A=\{y\}$.

By induction, if $\dim V_T<\dim V$ then $V_T$ 
has only finitely many optimal subvarieties.
We are reduced
to the case $\dim V_T=\dim V$, which is the case that $T$
is the family of points. We take a finite extension field $L$ of
$K$ over which all optimal subvarieties of 
positive dimension are defined.

Now suppose that there is a special subvariety $S$
intersecting $V$ optimally in a point $\{y\}$. 
Let $M^R$ be the family of
Mobius subvarieties containing the components of $\pi^{-1}(S)$
(note that the family of all Mobius subvarieties is definable).
Let ${\defZ}\subset M_R\times Y(1)^n$ be the (definable)
set of pairs $(t,u)$ such that $\{u\}$ is an optimal component
of $V\cap \pi(M^R_t\cap\FF_0^n)$. 
Let $\kappa$ be the constant afforded by LGO for $V$.
We apply Corollary 7.2 (with $\ell=0$) and $\epsilon=(20\kappa)^{-1}$
to get $c=c(\defZ, 2, \epsilon)$. Let $\Sigma=\polyt{\defZ}{2,T}$.

Let $T\ge 1$ and suppose $\# \pi_2(\Sigma)>cT^\epsilon$.
%
Then we have a curve $\beta$ in one of the constituent blocks
such that $\pi_1\circ\beta$ is semialgebraic and $\pi_2\circ\beta$
is non-constant.
The union of the Mobius subvarieties over the
complexification of $\pi_1\circ\beta$ meets $\pi^{-1}(V)$ in
an uncountable set, hence in a set of complex dimension at least
one. This union of Mobius varieties is thus 
a larger algebraic subvariety of $\HH^n$ with the same 
defect (in the sense of 5.11) as each fibre. 
By WCA there is a geodesic component containing this subvariety
of the same defect. The weakly special subvariety
is however special, as it contains some of the special subvarieties
(conjugates of $S$)
corresponding to $\{y\}$ and its conjugates. 
This contradicts the assumption
that $\{y\}$ and its conjugates are optimal.

Therefore $\# \pi_2(\Sigma)\le cT^\epsilon$. But if
$T=\Delta(\langle y\rangle)^{10}$ we have
$\# \pi_2(\Sigma)\ge T^{2\epsilon}/(2[L:K])$ by LGO and 
Proposition 6.8. Therefore $\Delta(\langle y\rangle)$ is bounded.
%
This completes the proof of Theorem 
\ref{thm:LGOWCAimplyZPrefined}.
\end{proof}

\section*{Acknowledgements}
The authors thank Martin Bays and Martin Hils for informing them
about the work of Poizat and Fabien Pazuki for comments on the text.
PH was partially supported by the National Science Foundation under
agreement No. DMS-1128155. Any opinions, findings and conclusions or
recommendations expressed in this material are those of the authors
and do not necessarily reflect the views of the National Science
Foundation. 
JP acknowledges with thanks that his research 
was supported in part by a grant  from the EPSRC entitled
``O-minimality and diophantine geometry'', reference
EP/J019232/1.

\bigbreak

\bibliographystyle{amsplain}

\def\cprime{$'$}
\providecommand{\bysame}{\leavevmode\hbox to3em{\hrulefill}\thinspace}
\providecommand{\MR}{\relax\ifhmode\unskip\space\fi MR }
\providecommand{\MRhref}[2]{%
  \href{http://www.ams.org/mathscinet-getitem?mr=#1}{#2}
}
\providecommand{\href}[2]{#2}

\vfill

\noindent{Fachbereich Mathematik\hfill Mathematical Institute}\\
{Technische Universit\"at Darmstadt \hfill University of Oxford}\\
{64289 Darmstadt\hfill Oxford OX2 6GG}\\
{Germany\hfill UK}

\begin{center}
  \today
\end{center}

\end{document}